\documentclass[letter,11pt, reqno]{amsart}
\usepackage[margin=3cm]{geometry}

\usepackage{multicol}

\usepackage{adjustbox}
\usepackage{graphicx}%
\usepackage{multirow}%
\usepackage{amsmath,amssymb,amsfonts}%
\usepackage{amsthm}%
\usepackage{mathrsfs}%
\usepackage{xcolor}
\usepackage{booktabs}%
\usepackage{natbib}
\usepackage{hyperref}
\usepackage{caption}
\usepackage{amssymb}
\usepackage{amsfonts}
\usepackage{bm}
\usepackage{enumerate}
\usepackage{enumitem}
\newcommand{\R}{\mathbb{R}}

\newcommand{\N}{\mathbb{N}}
\newcommand{\B}{\mathbb{B}}
\newcommand{\E}{\mathbb{E}}

\newcommand{\D}{\mathrm{D}}

\newcommand{\BB}{\mathcal{B}}
\newcommand{\U}{\mathcal{U}}
\newcommand{\HH}{\mathcal{H}}
\renewcommand{\SS}{\mathbb{S}}
\renewcommand{\O}{\mathcal{O}}

\newcommand{\dsum}{\displaystyle\sum}

\newcommand{\dsup}{\displaystyle\sup}

\newcommand{\conv}{{\rm conv}}

\renewcommand{\;}{\quad}

\theoremstyle{thmstyleone}%
\newtheorem{theorem}{Theorem}

\newtheorem{proposition}[theorem]{Proposition}%
\newtheorem{ex}[theorem]{Example}%
\newtheorem{lemma}[theorem]{Lemma}%
\newtheorem{corollary}[theorem]{Corollary}%

\theoremstyle{thmstyletwo}%
\newtheorem{remark}{Remark}%

\theoremstyle{thmstylethree}%
\newtheorem{definition}{Definition}%

\definecolor{punkpink}{RGB}{255, 20, 147}

\definecolor{vbcolor}{rgb}{1.0, 0.6, 0.4}

\usepackage{amsmath,amssymb,amsthm,mathtools}
\usepackage[ruled,vlined]{algorithm2e}
\usepackage{bm}

\usepackage{tabularx}
\usepackage{booktabs} 
\usepackage{multirow}
\usepackage{booktabs}
\let\origmaketitle\maketitle
\def\maketitle{
	\begingroup
	\def\uppercasenonmath##1{} 
	\let\MakeUppercase\relax 
	\origmaketitle
	\endgroup
}

\begin{document}

\title[]{\Large Structural and Solution Analysis for the Ordered Weber Problem under Spatial Uncertainty}

\author[V. Blanco, D. Laborda \MakeLowercase{and} M. Mart\'inez-Ant\'on]{
{\large V\'ictor Blanco$^{\dagger}$ and  Miguel Mart\'inez-Ant\'on$^{\dagger}$}\medskip\\
$^\dagger$Institute of Mathematics (IMAG), Universidad de Granada\\
\texttt{vblanco@ugr.es}, \texttt{mmanton@ugr.es}
}

\maketitle

\begin{abstract}
    We propose a general analytical framework for single-facility continuous location problems under spatial demand uncertainty. In contrast to classical formulations based on discrete or regionally aggregated demands, the proposed model represents uncertainty through general probability measures on $\R^d$, thereby encompassing finite, bounded, and unbounded support distributions within a unified formulation. The objective aggregates expected distances by means of an ordered weighted averaging operator, providing a flexible mathematical structure that includes the classical Weber problem and its ordered extensions as special cases. We establish fundamental properties of this stochastic ordered Weber model, including convexity, continuity, and existence of optimal solutions, and we derive quantitative bounds on the proximity between stochastic minimizers and the convex hulls of demand supports. Building upon these results, we develop and analyze an adaptive sample average approximation scheme, proving its convergence and deriving finite-sample error estimates under mild regularity conditions. For spherically symmetric distributions, we further obtain explicit analytical expressions for the approximation error. Together, these results provide a rigorous mathematical foundation for a broad class of stochastic ordered location models and highlight new theoretical connections between convex analysis, stochastic programming, and ordered optimization.
\end{abstract}



\section{Introduction} \label{sec:intro}
The Weber problem is a foundational model in continuous location theory and mathematical optimization. It seeks a point in Euclidean space that minimizes the sum of weighted distances to a finite set of demand points~\citep{Weber1909,Love1988,Brimberg2003,chandrasekaran1990algebraic}. From an optimization perspective, the Weber problem represents a prototypical example of a convex yet non-differentiable minimization problem, as its objective function is convex but exhibits singularities on the demand points. This structure has made the model a benchmark for the development and analysis of algorithms in non-smooth optimization, most notably Weiszfeld’s classical iterative method and its numerous extensions~\citep{Weiszfeld1937,Brimberg1995}. Moreover, the Weber formulation serves as a continuous analogue and theoretical precursor to several discrete and network-based models, such as the $p$-median and $p$-center problems, thereby bridging geometric optimization, facility location, and data-driven modeling~\citep{DreznerHamacher2002,Small1990,Cardinal2009}. 

Over the last decades, a wide range of strategies have been developed to solve the Weber problem efficiently. While early studies mainly focused on Euclidean distances in the plane~\citep[see, e.g.,][]{Love1988,Brimberg2003}, leading naturally to geometric and iterative approaches, more recent works have exploited modern convex optimization tools. In particular, formulations based on $p$-order cone programming enable an efficient treatment of extended Weber-type problems in higher dimensions, with general $\ell_p$-norm distances, ordered objective functions~\citep{espejo2009convex,blanco2014revisiting}, or even geodesic distances~\citep{blanco2017continuous}.

In many real-world and theoretical settings, the exact locations of demand points are not known with certainty. This gives rise to models that incorporate spatial uncertainty or other stochastic components, reflecting imprecise data aggregation, temporal variability in demand, or inherent randomness in spatial distributions. To address such cases, several generalizations of the Weber problem have been proposed in which each demand is represented by a region, a probability distribution, or an uncertain set~\citep{carrizosa1995generalized,carrizosa1998weber,Krarup1978,DreznerWesolowsky1991,BrimbergDrezner2011,kalczynski2025weber}. These extensions, often referred to as \emph{Weber problems with regional demands} or \emph{stochastic Weber problems}, give rise to rich classes of convex and robust formulations. They have deepened the theoretical understanding of location models under uncertainty and inspired new algorithmic developments for minimizing expected or worst-case distance costs, reinforcing the Weber problem’s central role as a bridge between geometry, convex analysis, and optimization.

Despite these advances, the existing literature remains limited in generality. On the one hand, some recent studies~\citep[e.g.,][]{kalczynski2025weber} assume that the probability distributions representing regional demands are identical, implicitly imposing homogeneous stochastic behavior across all demands. On the other hand, most stochastic formulations in the literature describe uncertainty through compact regions, typically endowed with spatial density functions defining the likelihood of each demand location within its region. Yet, in many applications, the uncertainty in the location of demand points is neither bounded nor homogeneous, which motivates the need for a more general framework capable of handling unbounded, heterogeneous, or distributional uncertainty in a unified way.

Table~\ref{tab:review} summarizes the main contributions to the study of Weber-type problems with spatially uncertain demand, illustrating the progressive generalization of the classical point-based model. Early works, such as \citet{love1972x}, initiated the analysis of continuous demand by considering uniformly distributed rectangular regions under the Euclidean distance. Subsequent contributions by \citet{carrizosa1995generalized, carrizosa1998weber, carrizosa1998location} established a probabilistic and geometric framework allowing both demand and facilities to occupy extended regions and generalized the model to higher dimensions and arbitrary gauge-based distances. Later studies focused on specific geometries and computational strategies: \citet{Fekete2005} addressed polygonal domains and geodesic paths under $\ell_1$ norms, \citet{valero2008single} proposed Weiszfeld-like algorithms for general $\ell_p$ norms with $p \in [1,2]$, whereas \citet{kalczynski2025weber} and \citet{ByrneKalcsics2022} analyzed disc-shaped and barrier-constrained regions, respectively, under the Euclidean metric. \citet{Yao2014} introduced a GIS-based discretization approach to approximate continuous regional demand. Finally, \citet{puerto2011structure} provided an efficient algorithm to characterize the entire set of optimal solutions for the case of total polyhedrality (in terms of both metrics and demand regions) in the plane. Additionally, they developed a discretization result that yields $\epsilon$-approximate solutions for the more general case.
\begin{table}[ht]
\centering
\renewcommand{\arraystretch}{1.1}
\setlength{\tabcolsep}{4pt}
\begin{adjustbox}{width=\textwidth,center}
\begin{tabular}{c c p{1.4cm} p{1.7cm} p{9cm} l}
\toprule
\textbf{Shape} & \textbf{Dim.} & \textbf{Demand} & \textbf{Metric} & \textbf{Contribution} & \textbf{Reference} \\
\midrule
R & 2D & Uniform & $\ell_2$ & First analytical treatment of the continuous Weber problem with \linebreak uniformly distributed regional demand. & \cite{love1972x} \\
G & 2D & General & Gauge & Introduces a unified probabilistic framework where both demand and facility are extended regions. & \cite{carrizosa1995generalized} \\
G & 2D & General & Gauge & Defines the regional Weber problem and proposes approximation methods for expected distances. & \cite{carrizosa1998weber} \\
G & $n$D & General & Gauge & Jointly optimizes facility location and shape, extending the Weber model to regional facilities. & \cite{carrizosa1998location} \\
P & 2D & Uniform & $\ell_1$ & Provides exact algorithms for continuous demand over polygonal domains with obstacles. & \cite{Fekete2005} \\
H & $n$D & Uniform & $\ell_p$ & Weiszfeld-like iterative algorithm with convergence guarantees for \linebreak $\ell_p$-norm based distances with $p \in [1,2]$. & \cite{valero2008single} \\
P & 2D & General & Polyhedral & Provides an efficient algorithm to characterize the entire optimal \linebreak solution set. & \cite{puerto2011structure} \\
G & 2D & General & $\ell_2$ & GIS-based discretization of the problem. & \cite{Yao2014} \\
G & 2D & Uniform & $\ell_2$ & Models continuous demand in regions with forbidden zones or \linebreak barriers. & \cite{ByrneKalcsics2022} \\
D & 2D & Uniform & $\ell_2$ & Models spatially extended demand via discs and derives geometric properties and algorithms. & \cite{kalczynski2025weber} \\
\bottomrule
\end{tabular}
\end{adjustbox}
\caption{Representative contributions to the Weber problem with spatially uncertain demand. The shape of the demand regions is denoted as: D (discs), R (rectangles), G (general), P (polyhedra), and H (hypercubes).}
\label{tab:review}
\end{table}

The goal of this paper is to propose a general framework for Weber-type problems that does not rely on predefined demand regions, but instead models \emph{spatial uncertainty} through general probability measures. Such a formulation can easily be particularized to regional demands by bounding the probability support. The model accommodates arbitrary distance functions on $\mathbb{R}^d$ and introduces a flexible family of objective functions that aggregate the weighted expected distances from the demands to the facility, based on an \emph{ordered weighted} operator. These operators have been recognized as alternative aggregation mechanisms that can better capture realistic or multi-criteria preferences, as well as versatile schemes that generalize numerous classical measures by reordering costs according to their rank and weighting them by their ordered positions~\citep{Yager1988}. By adjusting the weight vector, ordered weighted functions can reproduce the minimum, maximum, median, quantile, or arithmetic mean, making them a powerful modeling tool for fairness-oriented models. In location science, ordered operators underpin the family of \emph{ordered median location problems}, which unify diverse cost-based objectives under a single formulation~\citep[see, e.g.,][]{g:puerto2000geometrical,nickel2005location,blanco2014revisiting,blanco2016continuous,Labbe2017,Marin2020,blanco2023fairness,Ljubic2024,espejo2009convex}.

We begin by establishing several theoretical properties of the proposed model, with particular attention to the structural behavior of its optimal solutions and their proximity to the convex hull of compact regions that concentrate most of the demand distribution. Building upon this analysis, we develop, for the first time in continuous location, an adaptive \emph{sample average approximation} (SAA) scheme specifically designed for the stochastic ordered Weber framework. We prove its convergence under mild regularity conditions, thereby providing a solid theoretical foundation for its use in stochastic continuous location. Moreover, for the case of spherically symmetric demand distributions, we derive analytical bounds that quantify the approximation error when the stochastic problem is replaced by its deterministic surrogate, obtained by substituting each random demand with its symmetry center. For certain distributional families, these bounds admit closed-form expressions, revealing the dependence of the approximation quality on geometric and probabilistic parameters. Finally, the proposed framework is complemented by computational experiments on benchmark instances, which illustrate the accuracy and computational efficiency of the method.

\paragraph{\bf Main Contributions}

The main contribution of this work is theoretical, supported by a computational validation on benchmark instances. The key advances can be summarized as follows:
\begin{enumerate}
\item We introduce the \emph{ordered Weber problem under spatial uncertainty}, in which demand is represented through arbitrary probability measures, and the distances are aggregated with ordered operators. This model unifies a broad family of continuous location models to uncertain spatial settings.
    \item We establish the \emph{convex-analytic structure} of the ordered Weber problem under spatial uncertainty, providing new results on continuity, convexity, coercivity, and compactness of the objective function and feasible set.
    \item We derive \emph{quantitative proximity bounds} between the stochastic minimizer and the convex hull of compact regions that concentrate most of the demand probability mass, thus revealing the geometric behavior of optimal solutions.
    \item We develop and analyze an \emph{adaptive sample average approximation scheme}, proving its convergence under mild regularity assumptions and characterizing its stability through finite-sample surrogates.
    \item For \emph{spherically symmetric distribution families}, we derive explicit analytical error bounds that yield closed-form characterizations of the deterministic approximation error arising when the continuous demand is replaced by the symmetry centers of the distributions.
    \item Finally, we validate the theoretical findings through \emph{computational experiments on benchmark instances}, which confirm the accuracy, efficiency, and practical interpretability of the proposed framework.
\end{enumerate}
\paragraph{\bf Organization}
The remainder of this paper is organized as follows. Section~\ref{sec:problem} introduces the general formulation of the ordered Weber problem with spatially uncertain demand, including its probabilistic setting, distance functions, an ordered-based objective. Section~\ref{sec:saa} presents the adaptive sample average approximation algorithm and establishes its convergence guarantees. Section~\ref{sec:sym} focuses on spherically symmetric demand distributions, deriving analytical error bounds for their deterministic approximations. Section~\ref{sec:comp} reports the results of our numerical validation experiments, illustrating the performance and robustness of the proposed approach. Finally, Section~\ref{sec:conc} concludes the paper and outlines future research directions.
\section{Preliminaries}  \label{sec:problem}
In this section, we introduce the notation used throughout the paper and formally state the ordered Weber problem with spatially uncertain demands.

Given $n\in \N$, we denote by $[n]:=\{1,\ldots,n\}$ the index set with cardinality $n$. For every $i \in [n]$, let $(\Omega_i, \mu_i)$ be a probability space in $(\R^d, \BB)$, where $\BB$ stands for the Borel $\sigma$-algebra over the Euclidean space $\R^d$. Let $X_i\in \Omega_i$ be a random vector drawn according to $\mu_i$ for every $i\in [n]$. Each of these random vectors represents a random demand in $\R^d$, and each of these demands, $X_i$, is endowed with a nonnegative weight $\omega_i\in \R_+$, for all $i \in [n]$, that represents the importance, preference, or population of the $i$th demand.

We are given a distance in $\R^d$, $\D: \R^d \times \R^d \rightarrow \R_+$. The goal of this paper is to find a point in $\R^d$ minimizing some loss function of the \emph{expected} distances between this point and the random demands $X_1, \ldots, X_n$. Note that the expression for each of the expected values is as follows:
\begin{equation*}
    \mathbb{E}\left[\D(y,X_i)\right] = \int_{x \in \Omega_i} \D(y,x) d\mu_i = \int_{\Omega_i} \D(y,x) f_i(x) dx, \quad \forall i \in [n],
\end{equation*}
where the last equality holds only in the case that there is an explicit density function $f_i: \Omega_i \to \R$ of each probability distribution $X_i\sim \mu_i$. Hereinafter, to improve readability, when computing expectations, one may integrate over the entire domain $\R^d$, under the assumption that the measure of the complement of $\Omega_i$ is zero, i.e., $\mu_i(\R^d \setminus \Omega_i) = 0$, and the density function $f_i$ vanishes there, i.e., $f_i(\R^d\setminus \Omega_i)=\{0\}$ for all $i\in [n]$.

Each of the expected values is aggregated by means of an ordered weighted sum that we define as follows.
\begin{definition}[Ordered Weighted Sum]
    Let $\bm{\lambda} = (\lambda_1, \ldots, \lambda_n)\in \R^n$ be a weight vector and $\bm{d}=(d_1,\ldots, d_n)\in \R^n$ be a cost vector. The \emph{ordered weighted sum} is defined as:
\begin{equation*}
    \O_{\bm{\lambda}}(\bm{d}):=\dsum_{i=1}^n\lambda_i d_{(i)},
\end{equation*}
where $(\cdot)\in \mathscr{S}_n$ is any element of the symmetric group on $n$ letters sorting the cost vector $\bm{d}$ in non-increasing order, i.e., $d_{(i)}\geq d_{(i+1)}$ for all $i\in [n-1]$.
\end{definition}

The goal of the \emph{ordered Weber problem under spatial uncertainty} is to find the placement of a center, $y \in \R^d$, that minimizes the ordered weighted sum of the expected distances from $y$ to each random vector $X_i$, i.e.,
\begin{equation}\label{eq:owp}\tag{\rm OWP-SU}
    \min_{y \in \R^d} \O_{\bm{\lambda}}\left(\omega_{1}\mathbb{E}\left[\D\left(y,X_{1}\right)\right],\ldots, \omega_{n}\mathbb{E}\left[\D\left(y,X_{n}\right)\right]\right).
\end{equation}
Let $\rho(y) := \O_{\bm{\lambda}}\left(\omega_{1}\mathbb{E}\left[\D\left(y,X_{1}\right)\right],\ldots, \omega_{n}\mathbb{E}\left[\D\left(y,X_{n}\right)\right]\right) = \sum_{i=1}^n \lambda_i \omega_{(i)} \E\left[\D(y,X_{(i)})\right]$  denote the stochastic ordered Weber objective. For simplicity, we omit the explicit dependence on ${\bm \omega}$ and ${\bm \lambda}$, which are treated as fixed parameters of the problem.

We begin by analyzing the convex-analytic properties of this function.
\begin{proposition}
If $\D(\cdot,X)$ is convex and coercive for every $X$ and  $\lambda_1 \geq \cdots \geq \lambda_n \ge 0$, then $\rho$ is convex, coercive, and locally Lipschitz on $\R^d$.
\end{proposition}
\begin{proof}
The convexity and coercivity of $\rho$ follow from the convexity and coercivity of $\D$ and the convexity, monotonicity, and unboundedness of the ordered weighted sum $\O_{\bm\lambda}$~ \cite[see][]{grzybowski2011ordered}. 
Lipschitz continuity follows from the bounded subgradients of $\D$.
\end{proof}
These properties ensure the existence of minimizers and allow the convergence arguments of the solution approach proposed in the next section.

In the case where the random vectors $X_1, \ldots, X_n$ have finite support, with ${\rm supp}(\mu_i) = \{x_{i1}, \ldots, x_{im_i}\}$ for all $i \in [n]$, the problem simplifies to a discrete ({\rm d}) version, where the integrals are replaced by finite averages:
\begin{equation}\label{eq:omp_discrete}\tag{\rm OWP-SU$_{\rm d}$}
\min_{y \in \R^d} \sum_{i=1}^n \lambda_i \frac{\omega_{(i)}}{m_{(i)}} \sum_{j=1}^{m_{(i)}} \D(y, x_{(i)j}).
\end{equation}
This version of the problem has been widely studied in the literature as the continuous ordered median problem~\citep[see, e.g.,][]{blanco2014revisiting}. In the following result, we provide a valid convex formulation for the problem in case the $\bm{\lambda}$-weights are nonnegative and sorted in non-increasing order (the so-called \emph{convex} ordered median problem).
\begin{proposition}
    If $\D(\cdot,X)$ is convex and $\lambda_1 \geq \cdots \geq \lambda_n \geq 0$. Then, \ref{eq:omp_discrete} can be solved with the following convex program:
    \begin{align}
        \text{minimize} & \; \sum_{i=1}^n u_i + \sum_{j=1}^n v_j\nonumber\\
        \text{subject to} & \;  u_i + v_j \geq \lambda_j \frac{\omega_{i}}{m_{i}} \sum_{\ell=1}^{m_{i}} \D(y,x_{i\ell}), \; \forall i, j \in [n];\label{eq:sort}\\
        & \; u, v \in \R^n_+;\nonumber\\
        & \; y \in \R^d.\nonumber
    \end{align}
\end{proposition}
\begin{proof}
    The result follows by adapting to our settings the reformulation of convex ordered weighted sums proposed by \cite{blanco2014revisiting}.
\end{proof}
Although geometric algorithms have been proposed for specific instances of this problem (e.g., with Euclidean distances, equal weights, or planar demand), modern convex optimization techniques have proven highly effective for solving it efficiently with off-the-shelf solvers. In particular, when $\D$ is a polyhedral-norm-based distance, the problem can be formulated as a linear program, whereas for $\ell_p$-norm-based distances it can be expressed as a $p$th-order cone program~\citep{blanco2014revisiting,blanco2024minimal}. Specifically, in the $\ell_p$-norm case, the constraints in~\eqref{eq:sort} can be reformulated as:
\begin{align*}
u_i + v_j &\geq \lambda_j \frac{\omega_{i}}{m_{i}} \sum_{\ell=1}^{m_{i}} z_{i\ell}, & \forall i, j \in [n],\\
z_{i\ell} &\geq \|y - x_{i\ell}\|_p, &\forall i \in [n], \ell \in [m_i],
\end{align*}
that can be efficiently represented as a set of linear and second-order cone constraints~\citep{blanco2024minimal} and possesses a known controlled complexity~\citep{blanco2025complexity}.

In what follows, we provide a simple illustrative example that highlights the differences between solving the ordered Weber problem under deterministic and uncertain demand representations. In the deterministic setting, each demand distribution is replaced by its centroid, yielding the classical ordered Weber formulation. In contrast, the uncertain model accounts for the full spatial probability distribution of each demand region.
\begin{ex}
We consider one of the datasets introduced in~\cite{kalczynski2025weber}, where the authors assumed that each demand is uniformly distributed over a disc with a given center and radius. Figure~\ref{ex:1} displays the optimal solutions for two classical ordered models: the \texttt{median} (min-sum) and the \texttt{center} (min-max) problems, under different, generally non-uniform spatial demand distributions (with color intensity indicating cumulative probability). The symbol $\times$ denotes the solution of the uncertain model (also labeled as (\texttt{u}) in the legend), whereas $\blacksquare$ represents the solution to the deterministic approximation obtained by concentrating each demand at its center (labeled as (\texttt{d}) in the legend).
    \begin{figure}
\centering
\fbox{\includegraphics[width=0.45\textwidth]{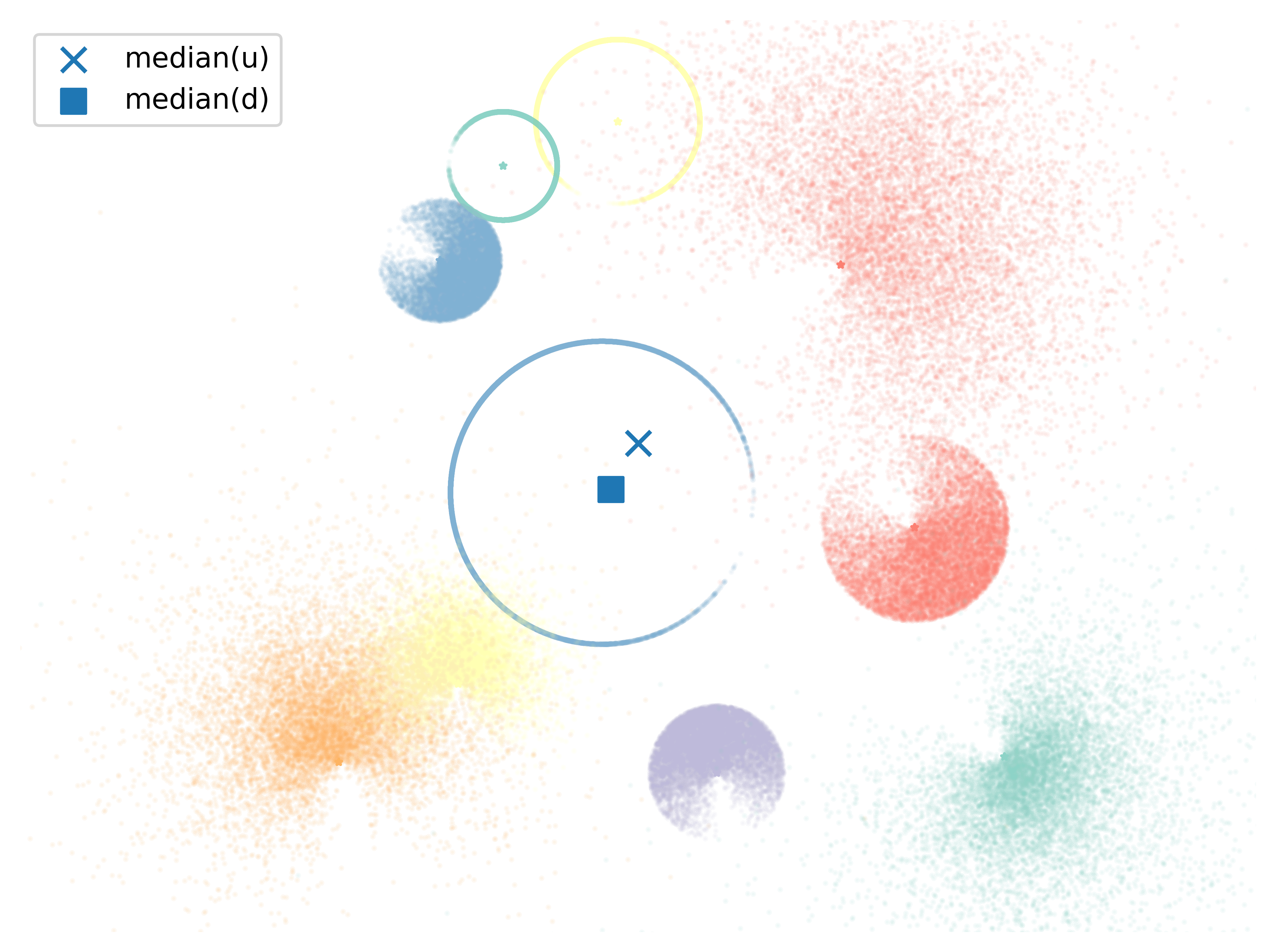}}~\fbox{\includegraphics[width=0.45\textwidth]{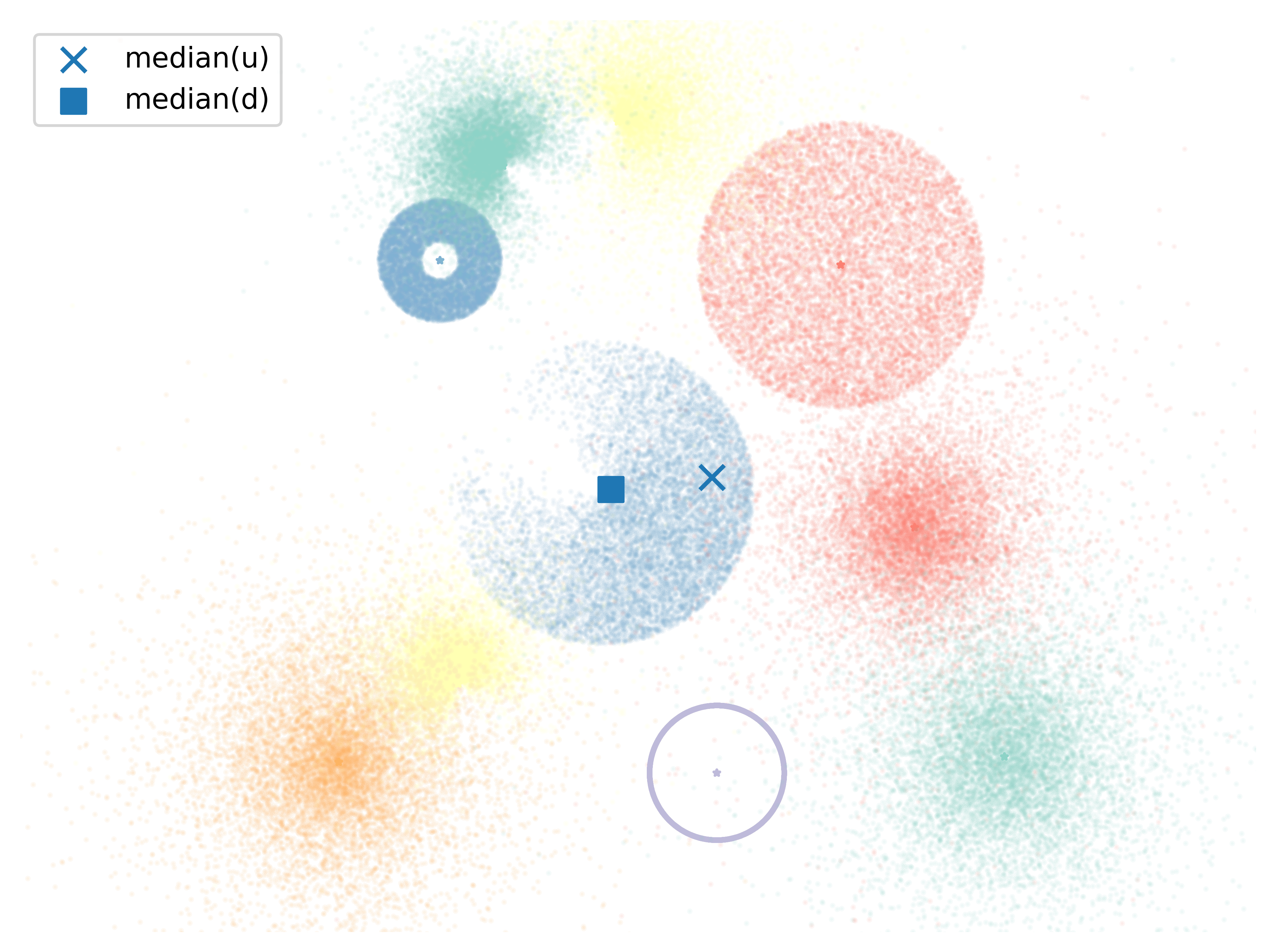}}\\
\fbox{\includegraphics[width=0.45\textwidth]{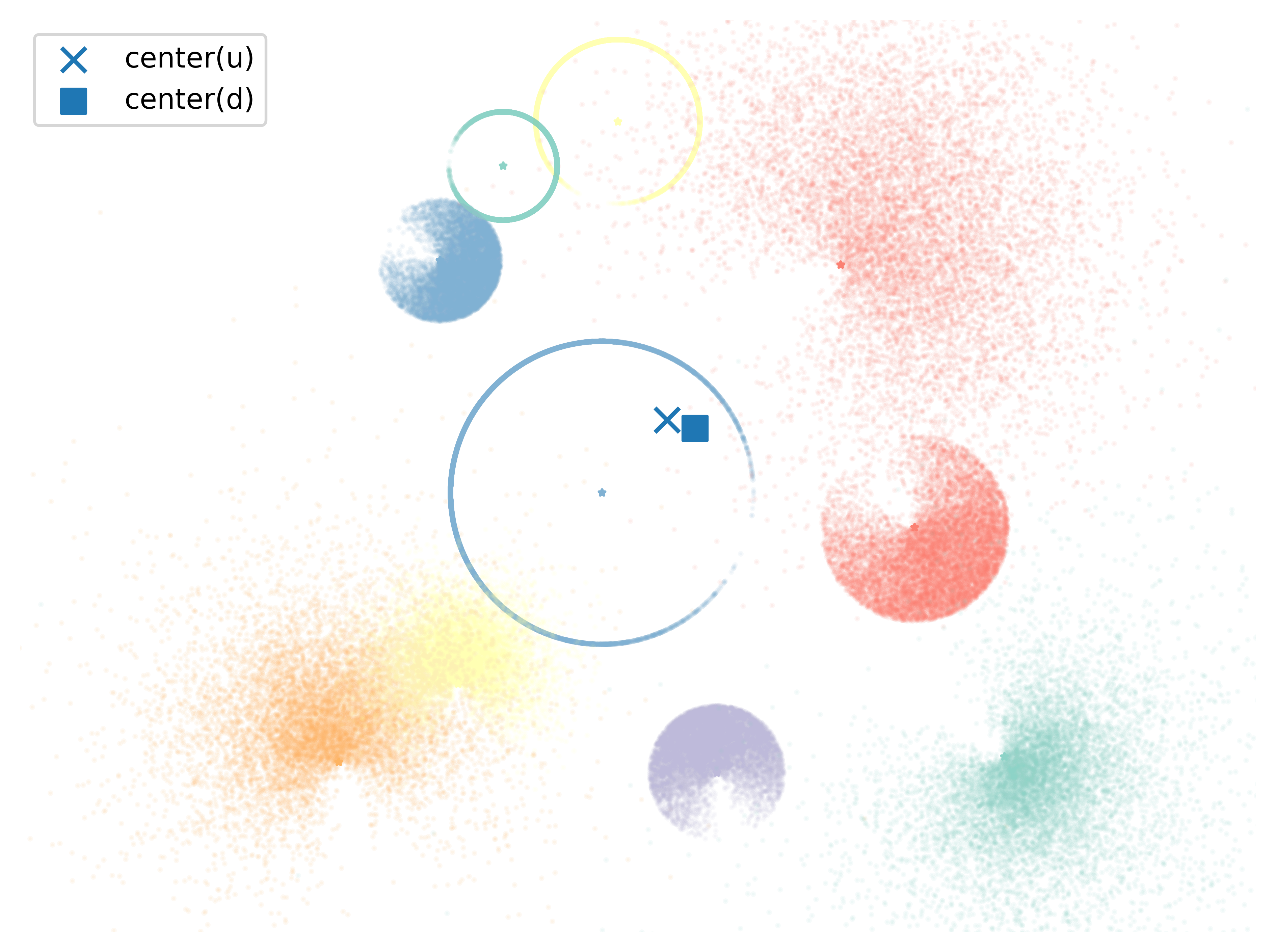}}~\fbox{\includegraphics[width=0.45\textwidth]{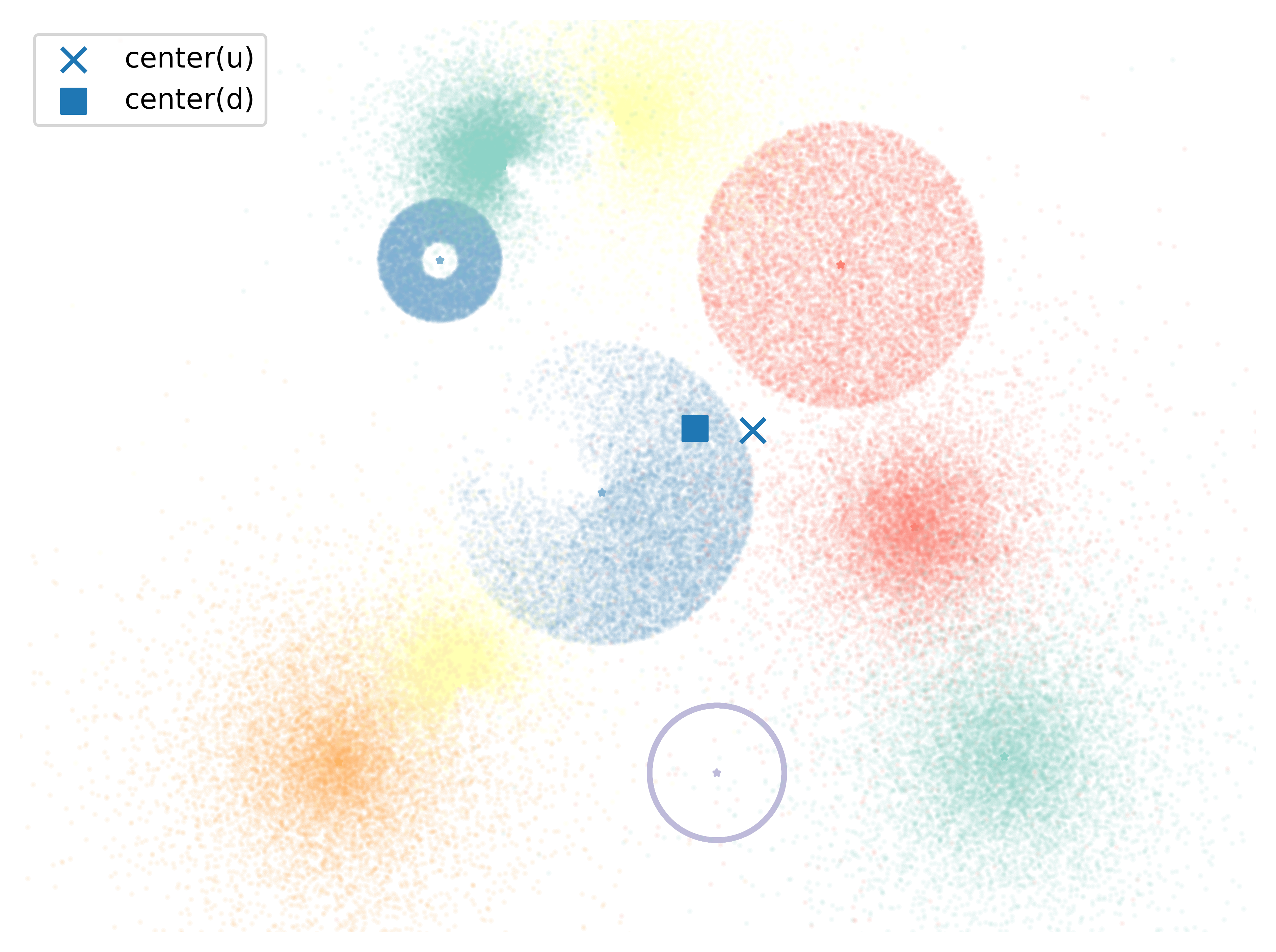}}
\caption{Solutions achieved with deterministic (\texttt{d}) and stochastic (\texttt{u}) versions of the \texttt{median} (top) and \texttt{center} (bottom) problems for different (non-symmetric non-uniform) demand distributions.\label{ex:1}}
    \end{figure}
One may notice that in each problem the two simplified deterministic versions are defined by the same parameters, and therefore they reach the same solution, as expected. In contrast, the stochastic problems capture the variability of the demand locations, balancing the solution according to the probability masses.
\end{ex}
The classical Weber problem, as well as its ordered extensions, verify that their solutions belong to the convex hull of the demand points. In the ordered Weber with spatially uncertain demand, in case the support of the probability measures for all the demands is compact (and then bounded), a similar result can also be derived. For the general case, we derive a structure result that provides information on how far an optimal solution to the problem is from a convex and compact set that concentrates a large mass of the probability for the demand.
\begin{proposition}\label{prop:eps-hull}
Let $X_i$ be a random vector drawn according to $\mu_i$ and $K_i\subset\R^d$ be a compact set such that
$\mu_i(K_i)\ge1-\varepsilon_i$, for every $i\in[n]$. Let $\D$ be a norm-based distance. Then, if $\bar\varepsilon :=  \max_{i\in [n]}\varepsilon_i<\frac{1}{2}$, every minimizer $y^\star$ of \ref{eq:owp} satisfies
\begin{equation*}
    \D\left(y^\star, K\right) \le \frac{\bar\varepsilon}{1-2\bar\varepsilon}\max\{\|z-w\|: z,w\in K\}, \text{ where } K:=\conv\left(\bigcup_{i=1}^n K_i\right).
\end{equation*}
\end{proposition}
\begin{proof}
Fix $y\notin K$. Let $\pi_K(y)$ be the closest point of $K$ to $y$ and $\delta:=\|y-\pi_K(y)\|=\D(y,K)>0$. By the supporting hyperplane theorem, there exist  $v$ with $\|v\|_*=1$ (here $\|\cdot\|_*$ stands for the dual norm of $\|\cdot\|$) such that $\langle v,y-\pi_K(y)\rangle=\delta$ and $\langle v,z-\pi_K(y)\rangle\le 0$ for all $z\in K$. Thereby, $\langle v,y-z\rangle=\langle v,y-\pi_K(y)\rangle+\langle v,\pi_K(y)-z\rangle\ge \delta$. Hence, on $X_i\in K$
\begin{equation}\label{eq:compact}
    \frac{\langle v,y-X_i\rangle}{\|y-X_i\|}\ge\frac{\delta}{\|y-X_i\|}\ge\frac{\delta}{\delta+ \max_{z, w \in K}\|z-w\|},
\end{equation}
where the last inequality comes from the triangle inequality using $\pi_K(y)$ as intermediary.
We have the following one-sided directional derivative bound
\begin{align*}
        \partial_{(y;-v)} \E[\|y - X_i\|]&:=\frac{d}{dt}\E\left[\|y-tv-X_i\|\right] & [t\to 0^+] \\
        &=\frac{d}{dt}\int_{x\in \R^d}\|y-tv-x\|d\mu_i & [t\to 0^+] \\
        & = \int_{x\in \R^d}\frac{d}{dt}\|y-tv-x\|d\mu_i & [t\to 0^+] \\
        & = -\int_{x\in \R^d}\frac{\langle v, y-x\rangle}{\|y-x\|}d\mu_i & \\
        & = -\int_{x\in K}\frac{\langle v, y-x\rangle}{\|y-x\|}d\mu_i - \int_{x\in \R^d\setminus K}\frac{\langle v, y-x\rangle}{\|y-x\|}d\mu_i &\\
        & \leq (\varepsilon_i-1)\frac{\delta}{\delta+\max_{z, w \in K}\|z-w\|}+\varepsilon_i,
\end{align*}
where the last inequality follows directly from \eqref{eq:compact}, the H\"older inequality, and $\mu_i(K)\geq 1-\varepsilon_i$.

If $\delta>\frac{\bar\varepsilon}{1-2\bar\varepsilon}\max_{z, w \in K}\|z-w\|$, then $\partial_{(y;-v)} \E[\|y-X_i\|]<0$ for every $i\in [n]$. 
Therefore, every component $\omega_i\E\left[\|y-X_i\|\right]$ strictly decreases along $-v$.

Finally, since the ordered weighted sum, $\O_{\bm\lambda}$, with nonnegative $\bm \lambda$-weights is monotone in each argument, then a uniform componentwise decrease in $\left(\omega_i\E[\|y-X_i\|]\right)_{i=1}^n$ implies a strict decrease of the objective function. This conclusion contradicts the optimality of any minimizer lying at distance $\delta > \frac{\bar\varepsilon}{1-2\bar\varepsilon}\max_{z, w \in K}\|z-w\|$ from $K$, and then any minimizer $y^\star$ must satisfy the stated bound.
\end{proof}
The above result allows us to ensure that, if the $\varepsilon_i$ are small, the solution to our problem is close to the convex hull of the (almost) support of the demands. If all the supports of the demands lie in compact sets, then $\bar{\varepsilon} = 0$, and the solution belongs to the convex hull of those supports. This situation covers most of the prior work on the topic, where the demands are assumed to be supported on compact regions.

\section{Sample Average Approximation and Convergence Analysis}\label{sec:saa}

This section is devoted to proposing an iterative solution scheme for solving \ref{eq:owp} and proving its convergence under mild assumptions.
\subsection*{Adaptive SAA Algorithm}
Note that the \ref{eq:owp} involves expectations with respect to random demand locations, whose exact evaluation is typically intractable even for very simple distributional forms. A natural and widely used approach to approximate such stochastic optimization problems is the \emph{sample average approximation} (SAA, for short) method~\citep{ShapiroHomemDeMello2000}, in which the expectations in the objective function are replaced by empirical finite averages computed from independent realizations of the underlying random vectors, as in \ref{eq:omp_discrete}. This section formalizes the SAA framework for the general problem and presents classical results that guarantee its asymptotic consistency and convergence properties under mild assumptions.

Given a collection of independent samples $\mathbf{A}=\{A_1, \ldots, A_n\}$, with 
$A_i = \{ x_{i1}, \ldots, x_{im_i} \}$ drawn according to $\mu_i$ for every $i\in[n]$, 
we define the SAA objective function for \ref{eq:owp} by
\begin{equation*}
\label{eq:saa}
\rho_\mathbf{A}(y)
:= \sum_{i=1}^n \lambda_i
\frac{\omega_{(i)}}{m_{(i)}}
\sum_{j=1}^{m_{(i)}} \mathrm{D}(y, x_{(i)j}),
\end{equation*}
where $\frac{\omega_{(i)}}{m_{(i)}}
\sum_{j=1}^{m_{(i)}} \mathrm{D}(y, x_{(i)j}) \geq \frac{\omega_{(i+1)}}{m_{(i+1)}} \sum_{j=1}^{m_{(i+1)}} \mathrm{D}(y, x_{(i+1)j})$, for all $i \in [n-1]$.

The solution of the problem under such an objective function is denoted as
\begin{equation*}
\label{eq:saa-problem}
\rho^\star_\mathbf{A} := \min_{y \in \mathbb{R}^d} \rho_\mathbf{A}(y),
\end{equation*}
and $y^\star_\mathbf{A}$ stands for the SAA minimizer of the collection of samples $\mathbf{A}$.

\SetKwInput{KwIn}{Input}
\SetKwInput{KwOut}{Output}
\SetKwInput{KwData}{Data}
\SetKwInput{KwInit}{Initialize}
\SetKw{Break}{break}
\SetKw{KwTo}{to}
\begin{algorithm}[ht]
\caption{Adaptive Sample Average Approximation (SAA)}\label{alg:saa}
\small\KwIn{Initial sample, $\mathbf{A}^{(0)}$ of sizes $\mathbf{m}^{0}=(m_1^0,\ldots,m_n^0)$, growth factor $\gamma>1$, tolerances $\varepsilon_1,\varepsilon_2>0$, maximum iterations $k_{\max}$, and validation sample $\mathbf{K}^{\rm val}$.}
\KwData{Distributions $(\mu_1,\ldots,\mu_n)$, weights $(\omega_1,\ldots,\omega_n)$.}

\For{$k \gets 0$ \KwTo $k_{\max}$}{
  \tcp*[h]{1) Draw training samples:} $A_i^{(k)} \gets \{x_{ij}^{(k)}\}_{j=1}^{m_i^k} \sim \mu_i$, $\mathbf{A}^{(k)} \gets \big\{A_i^{(k)}\big\}_{i=1}^n$.

  \tcp*[h]{2) Solve the SAA subproblem on $\mathbf{A}^{(k)}$:}  
  $\displaystyle y^{(k)} \in \arg\min_{y\in\mathbb{R}^d}\ \rho_{\mathbf{A}^{(k)}}(y)$. 
  
  \tcp*[h]{3) Bootstrap validation:} $\rho_{\mathbf{K}^{\rm val}}(y^{(k)})$,  $I^{(k)}\gets[\rho_{\mathbf{K}^{\rm val}}(y^{(k)}) \pm  r^{(k)}]$.

  \tcp*[h]{4) Per-group stability:}  
  
  \For{$i \leftarrow 1$ \KwTo $n$}{
    Compute contribution estimate $\rho_{A_i^{(k)}}(y^{(k)})$ and halfwidth $r_i^{(k)}$.
  }

  \tcp*[h]{5) Global stopping:}
  
  \If{$k>0$ \textbf{ and } $\forall i:\ \big| \rho_{A_i^{(k)}}(y^{(k)}) - \rho_{A_i^{(k-1)}}(y^{(k-1)}) \big| \le \varepsilon_{1}$ \textbf{ and } $r_i^{(k)} \le \varepsilon_{2}$}{
    $y^{\star} \gets y^{(k)}$,\quad $\rho^{\star} \gets \rho^\star_{\mathbf{A}^{(k)}}$,\quad $I^{\star} \gets I^{(k)}$ \;
    \Break
  }
  \tcp*[h]{6) Adaptive sample growth:}
  
  \For{$i \leftarrow 1$ \KwTo $n$}{
    \eIf{$\big| \rho_{A_i^{(k)}}(y^{(k)}) - \rho_{A_i^{(k-1)}}(y^{(k-1)}) \big| > \varepsilon_{1}$ \textbf{ or } $r_i^{(k)} > \varepsilon_{2}$}{
      $m_i^{k+1} \gets \left\lceil \gamma m_i^{k} \right\rceil$ \;
    }{
      $m_i^{k+1} \gets m_i^{k}$ \;
    }
  }
}
\If{$k>k_{\max}$}{
  $y^{\star} \gets y^{(k)}$,\quad $\rho^{\star} \gets \rho^\star_{\mathbf{A}^{(k)}}$,\quad $I^{\star} \gets I^{(k)}$ \;
}
\KwOut{Candidate solution $y^{\star}$, estimate ${\rho}^{\star}$, and confidence interval $I^{\star}$.}
\end{algorithm}
 
In Algorithm~\ref{alg:saa}, we show the pseudocode for the adaptive SAA that we propose to solve the ordered Weber problem with spatially uncertain demands, and 
where the expectation is approximated by empirical finite means over progressively refined sample sets. The procedure starts from an initial training sample $\mathbf{A}^{(0)}=\{A_i^{(0)}\}_{i=1}^n$ of sizes $\mathbf{m}^{0}=(m_1^0,\ldots,m_n^0)$ and a fixed validation sample $\mathbf{K}^{\rm val}$, and at iteration $k$ (i) solves the deterministic SAA subproblem $\min_y \rho_{\mathbf{A}^{(k)}}(y)$ to obtain $y^{(k)}:=y^\star_{\mathbf{A}^{(k)}}$, and (ii) evaluates its out-of-sample performance via the validation estimator $\rho_{\mathbf{K}^{\rm val}}(y^{(k)})$ together with a confidence interval $I^{(k)}=[\rho_{\mathbf{K}^{\rm val}}(y^{(k)})\pm r^{(k)}]$ (the halfwidth $r^{(k)}$ is obtained by a bootstrap on $\mathbf{K}^{\rm val}$). To monitor \emph{local} stability and target sampling effort, the algorithm tracks per-group contributions on the training sets, requiring for each $i\in[n]$ that the inter-iteration change $\big|\rho_{A_i^{(k)}}(y^{(k)})-\rho_{A_i^{(k-1)}}(y^{(k-1)})\big|$ and the associated halfwidth $r_i^{(k)}$ fall below tolerances $\varepsilon_1$ and $\varepsilon_2$, respectively. If these conditions hold simultaneously, the method returns $y^\star=y^{(k)}$, together with the empirical objective $\rho^\star=\rho^\star_{\mathbf{A}^{(k)}}$ and the validation-based interval $I^\star=I^{(k)}$. Otherwise, sampling is adaptively intensified only for unstable groups by a proportion $\gamma-1$, which concentrates computation where uncertainty remains. Using a held-out $\mathbf{K}^{\rm val}$ decouples selection and assessment, mitigating optimism in the SAA fit, whereas the groupwise stability test prevents premature stopping driven by aggregate cancellations. In practice, each SAA subproblem is deterministic and efficiently solvable (via conic programming) for fixed sample sizes, and the progressive targeted growth achieves accurate solutions with fewer total samples than a single-shot large SAA, while remaining flexible for heterogeneous distributions and weights, and scalable across problem dimensions.

\paragraph{\bf Bootstrap Validation Phase}

As mentioned above, in the proposed SAA-based approach, at each iteration, in the validation step, $B$ bootstrap replicates of the validation sample are chosen by resampling $K_i$ elements (with replacement) from the validation data, for each $i \in [n]$. The information retrieved in this sampling approach allows us to construct a confidence interval as follows.

Once the current solution $y^{(k)}$ is evaluated in the $\ell$th bootstrap function:
    \begin{equation*}
        \sum_{i=1}^n \lambda_i \frac{\omega_{(i)}}{K_{(i)}}\sum_{j =1}^{K_{(i)}} \mathrm{D}(y^{(k)}x_{(i)j}^{\ell}), \; \forall \ell\in [B],
    \end{equation*}
we construct $q_{\alpha/2}$ and $q_{1-\alpha/2}$, denoting the empirical quantiles of the objective attained by $y^{(k)}$ in the  $B$ bootstrap replicates corresponding to the lower and upper tails of level $\alpha$.
Then the $1-\alpha$ bootstrap quantile confidence interval for the true expected optimal value
is given by $
I_{1-\alpha}\bigl(y^{(k)}\bigr)
:=
\bigl[
    q_{\alpha/2},\,
    q_{1-\alpha/2}
\bigr]$.

The corresponding \emph{halfwidth} of the interval is defined as 
\begin{equation*}
    r^{(k)} := \max\Big\{\rho_{\mathbf{K}^{\rm val}}\bigl(y^{(k)}\bigr) - q_{\alpha/2},\, q_{1-\alpha/2} - \rho_{\mathbf{K}^{\rm val}}\bigl(y^{(k)}\bigr) \Big\},
\end{equation*}
and then, the interval 
$\bigl[\rho_{\mathbf{K}^{\rm val}}\bigl(y^{(k)}\bigr)\pm r^{(k)}\bigr]$
provides an approximate $1-\alpha$ confidence region for the
true expected objective value at $y^{(k)}$.
Consequently, the upper bound $\rho_{\mathbf{K}^{\rm val}}\bigl(y^{(k)}\bigr)+ r^{(k)}$ can be interpreted as a \emph{statistical certificate} on the true optimal value with coverage probability $1-\alpha$.
This upper bound serves as the validation-based stopping and comparison criterion in the progressive SAA procedure.

\paragraph{\bf Convergence Guarantees}

We adopt the following standard conditions in the study of convergence properties of SAA-based algorithms~\citep[see, e.g.,][]{ShapiroDentchevaRuszczynski2014}:
\begin{enumerate}[label={\rm (A\arabic*)}]
\item For each $i\in [n]$, $\{X_{ij}: j\in[m_i]\}$ are i.i.d.\ with law $\mu_i$, independent across $i$, and $m_i\to\infty$. \label{a1}
\item There exists $y_0\in\R^d$ with $\E\left[\D(y_0,X_i)\right]<\infty$ for all $i\in [n]$, and an integrable envelope $M(\xi)$ such that $\D(y,\xi) \le M(\xi)$ for all $y\in\R^d$,  e.g., $\D(y,\xi)\le a\|\xi\|+b(1+\|y\|)$ with $\E\left[\|\xi\|\right]<\infty$.\label{a2}
\item For each $\xi$, $y\mapsto \D(y,\xi)$ is continuous on $\R^d$.\label{a3}
\item $\rho(y)\to\infty$ as $\|y\|\to\infty$.\label{a4}
\item The minimizer $y^\star$ is unique.\label{a5}
\end{enumerate}

Assumptions \ref{a1}--\ref{a5} will ensure that our SAA-based approach provides consistent and stable estimates of the true optimal solution. \ref{a1} guarantees that the random samples used in the approximation are independent and identically distributed according to the true demand distributions, and that the sample size grows to infinity, allowing empirical averages to converge to their expectations by the law of large numbers. \ref{a2} ensures that the expected costs are finite and that the distance function does not grow too quickly with respect to the random data, preventing instability caused by large sample realizations. \ref{a3} requires that small changes in the facility location $y$ lead to small changes in the distance $\D(y,\xi)$, which guarantees that the objective function behaves smoothly and that limits and minimizations can be interchanged safely. \ref{a4}  ensures that the objective function grows unboundedly as $\|y\|\to\infty$, so that the problem admits at least one finite minimizer and solutions do not escape to infinity. Note that this condition holds for norm-based distances under $\E\left[\|X_i\|\right]<\infty$. Finally, \ref{a5} requires the optimal solution to be unique, which allows the sequence of minimizers of the subproblems solved at each iteration of our approach to converge to the true optimal point. In sum, these assumptions provide the theoretical foundation for the almost sure convergence of SAA-based estimators to the true solution of our problem.

For the sake of readability, in what follows, we will set the following notation. $\rho^\star$ and $y^\star$ denote the optimal value and an optimal solution of \ref{eq:owp}, respectively. Meanwhile, $\rho_{\mathbf{m}}(y)$ stands for the random variable $\rho_{\mathbf{A}}(y)$ where the sample collections $\mathbf{A}$ have fixed sizes $\mathbf{m}$, then we denote the random variable $\rho^\star_{\mathbf{m}}:=\rho^\star_{\mathbf{A}}$ and the random vector $y^\star_{\mathbf{m}}:=y^\star_{\mathbf{A}}$. 

With the above assumptions and notations, we got the following convergence result for Algorithm \ref{alg:saa}. 

\begin{proposition}\label{lem:ULLN}
Under \ref{a1}--\ref{a4}:
\begin{enumerate}
    \item $\dsup_{y\in\R^d}\big|\rho_\mathbf{m}(y)-\rho(y)\big|\xrightarrow{\min_i m_i} 0$ almost surely (a.s.). \label{prop:convergency1}
    \item $\rho^\star_\mathbf{m}\xrightarrow{\min_i m_i} \rho^\star$ almost surely, and every cluster point of $y^\star_{\mathbf{m}}$ is a $\rho$-minimizer. Additionally, if \ref{a5} holds, then $y^\star_{\mathbf{m}}\xrightarrow{\min_i m_i} y^\star$ almost surely.\label{prop:convergency2}
    \item If $\rho$ is differentiable at $y^\star$ and each $X_i$ has a continuous density in a neighborhood of $y^\star$ so that $\nabla^2\rho(y^\star)\succ0$, then, there exists a positive definite covariance matrix $\Sigma$ such that:
\begin{equation*}
    \sqrt{\sum_{i=1}^n m_i}(y^\star_{\mathbf m}-y^\star)\sim\mathcal N(0,\Sigma).
\end{equation*}
\label{prop:convergency3}
\end{enumerate}
\end{proposition}
\begin{proof}
We prove the three claims under \ref{a1}--\ref{a4} (and \ref{a5} where stated), by appealing to classical results in stochastic programming and M-estimation
\citep{ShapiroDentchevaRuszczynski2014,RockafellarWets1998}.
\begin{enumerate}
\item Fix a compact set $K\subset\R^d$. 
By assumptions \ref{a1}--\ref{a3}, the class 
$\{\xi \mapsto \D(y,\xi) : y\in K\}$ is pointwise measurable and dominated by the integrable envelope $M(\xi)$ from \ref{a2}. 
Hence, by a dominated (Glivenko–Cantelli) uniform law of large numbers (ULLN) for i.i.d. samples,
\begin{equation*}
    \sup_{y\in K}\Bigl|\rho_{\mathbf m}(y)-\rho(y)\Bigr|\xrightarrow{\min_i m_i}0 \quad \text{a.s.}
\end{equation*}
This follows by applying the scalar ULLN to each $i$, summing with weights $\omega_i$, and using independence across $i$ 
\citep[Th.~5.3]{ShapiroDentchevaRuszczynski2014}.
By level-boundedness \ref{a4}, the sublevel sets of $\rho$ are compact and, for large enough $\mathbf m$, contain the argmin sets of $\rho_{\mathbf m}$. 
Hence, the compact-$K$ ULLN extends to~\citep[see, e.g.][]{RockafellarWets1998}:
\begin{equation*}
    \sup_{y\in\R^d}\big|\rho_{\mathbf m}(y)-\rho(y)\big| \xrightarrow{\min_i m_i} 0 \quad \text{a.s.},
\end{equation*}
\item From \ref{prop:convergency1}. and \ref{a4}, the functions $\rho_{\mathbf m}$ epi-converge to $\rho$ almost surely.
Then, by the argmin continuity theorem for epi-convergence
\citep{RockafellarWets1998},
\begin{equation*}
    \rho_{\mathbf m}^\star\xrightarrow{\min_i m_i} \rho^\star \quad \text{a.s.}, 
\qquad
\text{and every cluster point of } y^\star_{\mathbf m} \text{ minimizes } \rho.
\end{equation*}

If the minimizer is unique \ref{a5}, then $y^\star_{\mathbf m}\xrightarrow{\min_i m_i} y^\star$ almost surely
\citep[Th.~5.4]{ShapiroDentchevaRuszczynski2014}.
\item Assume $\rho$ is differentiable at $y^\star$ with positive definite Hessian $\nabla^2\rho(y^\star)\succ0$, 
and each $X_i$ has a continuous density near $y^\star$ (ensuring interchange of differentiation and expectation).
The first-order optimality of $y^\star_{\mathbf m}$ yields $\nabla\rho_{\mathbf m}(y^\star_{\mathbf m})=0$.
Expanding around $y^\star$,
\begin{equation*}
    0 = \nabla\rho_{\mathbf m}(y^\star) 
      + \bigl[\nabla^2\rho(y^\star)+o_p(1)\bigr](y^\star_{\mathbf m}-y^\star).
\end{equation*}

By \ref{a1}--\ref{a3} and the multivariate central limit theorem,
\begin{equation*}
    \sqrt{\bar m}\nabla\rho_{\mathbf m}(y^\star)
\sim
\mathcal N(0,\Omega),
\qquad 
\Omega := \sum_{i=1}^n \omega_i^2{\rm Var}\big(\nabla_y\D(y^\star,X_i)\big),
\end{equation*}
where $\bar m:=\sum_i m_i$.
Applying Slutsky’s Theorem~\citep[see, e.g.][]{goldberger1970econometric} gives
\begin{equation*}
    \sqrt{\bar m}(y^\star_{\mathbf m}-y^\star)
\sim
\mathcal N\big(0,\Sigma\big),
\qquad 
\Sigma := \big[\nabla^2\rho(y^\star)\big]^{-1}\Omega\big[\nabla^2\rho(y^\star)\big]^{-1},
\end{equation*}
which is positive definite since $\nabla^2\rho(y^\star)\succ0$.
This is the standard M-estimation limit distribution
\citep[see, e.g.,][]{ShapiroDentchevaRuszczynski2014}.
\end{enumerate}
\end{proof}
\begin{remark}
Under mild regularity conditions, the assumptions \ref{a1}--\ref{a4} can be substantially simplified. Indeed, if each random vector $X_i$ has finite first moments, the distance $\D$ is induced by a norm $\|\cdot\|$, and there exist compact sets $K_i\subset\R^d$ such that $\mu_i(K_i)\ge 1-\varepsilon_i$ for all $i\in[n]$, as in Proposition \ref{prop:eps-hull}, then all the required regularity conditions are automatically satisfied. In this setting, \ref{a1} holds under standard i.i.d.\ sampling schemes, \ref{a2} follows directly from the finiteness of the first moments, and \ref{a3} and \ref{a4} are immediate properties of norm-based distances.

On the other hand, the uniqueness of the optimal solution of the problem \ref{a5} is more difficult to ensure. For instance, if $\D$ is induced by an $\ell_p$ norm with $1<p<\infty$ and all weights $\lambda_i$ are strictly positive and \emph{strictly monotone}, then on each region where the ordering is fixed, the objective function is a strictly convex positive combination of strictly convex functions, and the region changes only on measure-zero boundaries. This yields a \emph{unique} global minimizer under mild non-degeneracy (no identity of the distances along a nontrivial segment). If the $\bm \lambda$ contains ties or zeros, strict convexity can be lost and multiple minimizers may arise, uniqueness then requires additional genericity assumptions that exclude ordering ties at the optimum. 
\end{remark}
The above result establishes the asymptotic validity and statistical consistency of the proposed adaptive SAA approach. 
Under standard regularity assumptions \ref{a1}--\ref{a5}, the empirical objective function $\rho_{\mathbf{m}}$ converges uniformly to its population counterpart $\rho$, ensuring that both the optimal values and minimizers of the SAA problems provide consistent estimators of their true stochastic counterparts. 
In particular, \ref{prop:convergency1}. guarantees a uniform law of large numbers, \ref{prop:convergency2}. proves almost sure convergence of optimal values and solutions, and \ref{prop:convergency3}. establishes the asymptotic normality of the estimators, thereby quantifying the sampling-induced variability around the true optimizer $y^\star$. 
These results justify the use of Algorithm~\ref{alg:saa} as a statistically sound and computationally tractable method for solving stochastic optimization problems, where the progressive refinement of the sample sizes preserves consistency while adaptively allocating computational effort to the most uncertain components.

 Having established the basic assumptions that ensure existence and measurability of optimal solutions, we now move beyond asymptotic convergence and provide a quantitative assessment of the sampling error inherent to the SAA estimator. In particular, we derive a finite-sample bound that characterizes how the expected deviation between the true stochastic objective and its empirical counterpart decays with the sample size. Such estimates are fundamental to understanding the stability and reliability of this type of sample average approximations.
\begin{theorem}
\label{thm:finite-sample-epshull-detailed}
Assume \ref{a1}--\ref{a3} and \ref{a5} hold, $\D$ is a norm-based distance, and the conditions of Proposition~\ref{prop:eps-hull} are satisfied with 
$\bar\varepsilon:=\max_{i\in[n]}\varepsilon_i<\tfrac12$. 
Let $K:=\conv\big(\bigcup_{i=1}^n K_i\big)$ and 
\begin{equation*}
    r_{\varepsilon}:=\frac{\bar\varepsilon}{1-2\bar\varepsilon}\,\max\{\|z-w\|:z,w\in K\}. 
\end{equation*}

Define $\mathcal Y_\varepsilon:=\{y\in\R^d:\D(y,K)\le r_{\varepsilon}\}$.  
Let $\tilde y^\star_{\mathbf m}\in\arg\min_{y\in\mathcal Y_\varepsilon}\rho_{\mathbf m}(y)$ be any SAA minimizer restricted to $\mathcal Y_\varepsilon$. 
Then there exists a constant $\Delta_\varepsilon>0$ (depending only on $\mathcal Y_\varepsilon$ through its metric entropy) such that
\begin{equation*}
    \mathbb{E}\left[\big|\rho(\tilde y^\star_{\mathbf m})-\rho^\star\big|\right]
\le \frac{\Delta_\varepsilon\,L}{\sqrt{\min_i m_i}}\,.
\end{equation*}
\end{theorem}
\begin{proof}
By Proposition \ref{prop:eps-hull}, $y^\star$ belongs to the compact set $\mathcal Y_\varepsilon$. Hence $
\rho^\star=\min_{y\in\mathbb R^d}\rho(y)=\min_{y\in\mathcal Y_\varepsilon}\rho(y)$. 

Let $\tilde y^\star_{\mathbf m}\in\arg\min_{y\in\mathcal Y_\varepsilon}\rho_{\mathbf m}(y)$ be any SAA minimizer restricted to $\mathcal Y_\varepsilon$. Then
$$
\rho(\tilde y^\star_{\mathbf m})
\le \rho_{\mathbf m}(\tilde y^\star_{\mathbf m})+\varepsilon_{\mathbf m}
\le \rho_{\mathbf m}(y^\star)+\varepsilon_{\mathbf m}
\le \rho(y^\star)+2\varepsilon_{\mathbf m}
= \rho^\star+2\varepsilon_{\mathbf m},
$$
where $\varepsilon_{\mathbf m}:=\sup_{y\in\mathcal Y_\varepsilon}\big|\rho_{\mathbf m}(y)-\rho(y)\big|$. Since  $\rho^\star\le \rho(\tilde y^\star_{\mathbf m})$, we obtain that $
\big|\rho(\tilde y^\star_{\mathbf m})-\rho^\star\big|\le 2\,\varepsilon_{\mathbf m}$. 

Assumptions \ref{a1}--\ref{a3} ensure that $\rho$ and $\rho_{\mathbf m}$ are well defined and continuous; by hypothesis $\rho$ is $L$–Lipschitz on $\mathbb R^d$, and $\rho_{\mathbf m}$ is then also $L$–Lipschitz. Hence the centered process $y\mapsto \rho_{\mathbf m}(y)-\rho(y)$ is $2L$–Lipschitz on the compact metric space $(\mathcal Y_\varepsilon,\D)$. The  Rademacher-complexity bounds for Lipschitz empirical processes
on compact metric spaces \citep[see, e.g.,][]{ShapiroDentchevaRuszczynski2014} result in the bound:
$$
\mathbb E[\varepsilon_{\mathbf m}] \le\frac{\tilde\Delta_\varepsilon L}{\sqrt{\min_i m_i}}.
$$
where $\tilde\Delta_\varepsilon>0$ does not depend on $\mathbf m$ (but in how many small balls are needed to cover the set $\mathcal Y_\varepsilon$). Taking expectations in the optimality-gap bound and renaming $\Delta_\varepsilon:=2\tilde\Delta_\varepsilon$ gives
$$
\mathbb E\!\left[\big|\rho(\tilde y^\star_{\mathbf m})-\rho^\star\big|\right]
\;\le\; \frac{\Delta_\varepsilon\,L}{\sqrt{\min_i m_i}}.
$$
If one considers the unrestricted SAA minimizer $y^\star_{\mathbf m}\in\arg\min_{y\in\mathbb R^d}\rho_{\mathbf m}(y)$, the same bound holds upon either restricting the empirical problem to $\mathcal Y_\varepsilon$ (which does not alter $\rho^\star$ and keeps $y^\star$ feasible) or noting that, by uniform convergence on $\mathcal Y_\varepsilon$ and uniqueness \ref{a5}, $y^\star_{\mathbf m}\in\mathcal Y_\varepsilon$ with probability tending to $1$ as $\min_i m_i\to\infty$. 
\end{proof}
The above result extends the convergence analysis to settings where the probability distributions may have unbounded support. 
By exploiting the localization property established in Proposition~\ref{prop:eps-hull}, the analysis is restricted to a compact region $\mathcal Y_\varepsilon$ that concentrates almost all the probability mass. 
This formulation quantifies the rate at which the empirical objective approaches its true expectation under such localized conditions, showing that the approximation error decreases proportionally to $\sqrt{\min_i m_i}$, up to constants depending on the geometry of $\mathcal Y_\varepsilon$. 
Beyond its methodological implications, this result highlights how the spatial concentration of demand distributions affects the stability of the stochastic ordered Weber objective with respect to sampling noise, thus connecting the probabilistic structure of the model with its convex-analytic geometry.






\section{Spherically Symmetric Demands: Error Bounds}\label{sec:sym}

Within the framework of solving \ref{eq:owp}, where demand is continuously distributed over $\R^d$, we examine the suitability of approximating the problem by one defined on a finite set of representative demand points. When demand locations follow spherically symmetric distributions, we derive upper bounds on the approximation error resulting from replacing random demand vectors with the symmetry centers of their respective distributions and solving the deterministic ordered Weber problem on these representative points.

\begin{definition}[Spherical Symmetry]
    Let $X$ be a $d$-dimensional random vector on the probability space $(\Omega,\mu)$,  and $y^* \in \R^d$ be a point. $X$ is said to be \emph{spherically symmetric} centered at $y^*$ if $\mu$ is invariant under rotations centered in $y^*$.
\end{definition}

Note that if $X$ is spherically symmetric
 centered at $y^*\in\mathbb{R}^d$, then $X$ admits the representation
\begin{equation*}
    X =y^* + RU,
\end{equation*}
where $U$ is uniformly distributed on the unit sphere $\mathbb{S}^{d-1}_1(0)$
and $R\geq 0$ is a scalar random variable (the \emph{radius}) independent of $U$.
Hence, by symmetry,
\begin{equation*}
    \E \left[\|y^*-X\|\right]= \E\left[R\right] = \int_0^{\infty} rf_R(r)dr,
\end{equation*}
where $f_R$ is the density of $R$. Note that within this context, the distance should be the one induced by the Euclidean norm $\|z \|:=\sqrt{\langle z,z\rangle}$, $z\in \R^d$.

\begin{lemma}\label{lemma:single-va}
    Let $(\Omega,\mu)$ be a probability space and $y^*\in \R^d$ be a point. If $\mu$ is a spherically symmetric measure centered in $y^*$, then $y^*$ is optimal solution to \ref{eq:owp} for any $X$ drawn according to $\mu$.
\end{lemma}
\begin{proof}
    It will be sufficient to prove that for every $y\neq y^*$, the inequality $\E\left[\|y-X\|\right]\geq \E\left[\|y^*-X\|\right]$ holds for any $X$ drawn according to $\mu$. First, it begins with the one-dimensional case $d=1$. Let $y\in \R$ be a real, then we have
    \begin{equation*}
        \E\left[|y-X|\right]=\int_{-\infty}^y(y-x)d\mu + \int_{y}^{+\infty}(x-y)d\mu.
    \end{equation*}

    We can assume that $y>y*$. Then, we have
    \begin{equation*}
        \E\left[|y-X|\right] - \E\left[|y^*-X|\right] = \int_{-\infty}^{y^*}(y-y^*)d\mu +\int_{y^*}^{y} (y+y^*-2x) d\mu + \int_{y}^{+\infty} (y^*-y) d\mu.
    \end{equation*}
    Since $\mu$ is invariant under rotations centered in $y^*$, then $\int_{-\infty}^{y^*}(y-y^*)d\mu = \int_{y^*}^{+\infty}(y-y^*)d\mu$. Thus, we get
    \begin{equation*}
        \E\left[|y-X|\right] - \E\left[|y^*-X|\right] = 2\int_{y^*}^{y} (y-x) d\mu \geq 0.
    \end{equation*}
    The same procedure can be made for $y<y^*$.

    Now, we follow with the case $d\geq 2$. Note that we can write any $y\in \R^d$ as $y = ru+y^*$, for a scalar $r\geq 0$ and a unit vector $u$. Since $\mu$ is invariant under rotations centered in $y^*$, to prove the claim it suffices to show that for any fixed unit vector $u\in \SS^{d-1}_1(0)$, it follows that $r = 0$ is the minimizer of $\E\left[\|ru+y^*-X\|\right]$ in $\R_+$. Notice that $\E\left[\|ru+y^*-X\|\right]$ is a continuous function in $r$, since for every $\epsilon>0$ and for every nonnegative $r$ and $r'$ such that $|r'-r|<\epsilon$, we have
    \begin{align*}
        \left|\E\left[\|ru+y^*-X\|\right] - \E\left[\|r'u+y^*-X\|\right]\right|&=\left|\E\left[\|ru+y^*-X\|-\|r'u+y^*-X\|\right]\right|\\
        & \leq \|r'u-ru\|=|r'-r|<\epsilon.
    \end{align*}
    Hence, by the Newton-Leibniz formula, we have

    \begin{equation*}
        \E\left[\|su+y^*-X\|\right]-\E\left[\|y^*-X\|\right]=\int_0^s\frac{d}{dr}\E\left[\|ru+y^*-X\|\right]dr, \; \forall s>0.
    \end{equation*}
    Thus, to see the result, it is sufficient to show that
    \begin{equation*}
        \frac{d}{dr}\E\left[\|ru+y^*-X\|\right]>0, \; \forall r>0.
    \end{equation*}
    
    We know that
    \begin{align}
        \frac{d}{dr}\E\left[\|ru+y^*-X\|\right]&=\frac{d}{dr}\int_{x\in \R^d}\|ru+y^*-x\|d\mu \nonumber\\
        & = \int_{x\in \R^d}\frac{d}{dr}\|ru+y^*-x\|d\mu \nonumber\\
        & = \int_{x\in \R^d}\frac{\langle ru+y^*-x, u\rangle}{\|ru+y-x\|}d\mu \label{eq_exp_diff}.
    \end{align}
    Since $\mu$ is invariant under rotations centered in $y^*$, we know that a random vector $X$ drawn according to $\mu$ with $\|y^*-X\|=R$ is actually drawn according to the uniform distribution on the sphere $\SS^{d-1}_R(y^*)$, i.e., $X\sim \U(\SS_R^{d-1}(y^*))$. We evaluate \eqref{eq_exp_diff} in a fixed $r_0>0$. Let $\mu_{R}$ be the probability measure of the random variable $R=\|y^*-X\|$ and let $Z\sim \U(\SS_R^{d-1}(y^*))$. We have that
    \begin{align}
        \left.\frac{d}{dr}\E\left[\|ru+y^*-X\|\right]\right|_{r_0}& = \int_{x\in \R^d}\frac{\langle r_0u+y^*-x, u\rangle}{\|r_0u+y^*-x\|}d\mu \nonumber \\
        & = \int_0^{+\infty}\frac{\Gamma\left(\frac{d}{2}\right)}{2\pi^{\frac{d}{2}}R^{d-1}}\left(\int_{z\in\SS^{d-1}_R(y^*)}\frac{\langle r_0u+y^*-z, u\rangle}{\|r_0u+y^*-z\|}d\HH^{d-1}\right)d\mu_R,\label{eq:dobleint}
    \end{align}
    where $\HH^{d-1}$ stands for the $(d-1)$-dimensional Hausdorff measure. Now, we study the positivity of the inner integral in \eqref{eq:dobleint} by cases. In the first case, we consider $R\leq r_0$, and obtain
    \begin{equation*}
        \langle r_0u+y^*-z, u\rangle = r_0-\langle z-y^*,u\rangle \geq r_0-R\geq 0,
    \end{equation*}
    where the chain of inequalities holds tight if and only if $z = r_0u+y^*$. So we get that the inner integral in \eqref{eq:dobleint} is strictly positive when $R \in (0, r_0]$. 

    Otherwise, $R>r_0$. We define the random variable $\alpha\in[0, \pi]$ to be the angle between $z-y^*$ and $u$, and we let $\mu_\alpha$ be its associated measure. We also define the random variables $\beta(\alpha), \gamma(\alpha):[0, \pi ]\to [0, \pi ]$ to be the angles between $r_0u+y^* -z$ and $u$, and between $y^*-z$ and $r_0u+y^*-z$, respectively. It is easy to see that the three angles shape a triangle with vertices $y^*, z$, and, $r_0u+y^*$, so $\alpha+\beta(\alpha)+\gamma(\alpha)=\pi$, and therefore $\beta(\alpha)\leq \pi-\alpha$, being tight just when $\alpha=0,\pi$. Finally,

    \begin{align}
        \int_{z\in\SS^{d-1}_R(y^*)}\frac{\langle r_0u+y^*-z, u\rangle}{\|r_0u+y^*-z\|}d\HH^{d-1} & = \int_0^\pi\cos\beta(\alpha)d\mu_\alpha \label{eq:int1}\\
        &= \int_0^{\frac{\pi}{2}}(\cos\beta(\alpha)+\cos\beta(\pi-\alpha) )d\mu_\alpha\label{eq:int2}\\
        & > \int_0^{\frac{\pi}{2}}(\cos(\pi-\alpha)+\cos \alpha )d\mu_\alpha=0,\label{eq:int3}
    \end{align}
where \eqref{eq:int1} follows from scalar product formula, \eqref{eq:int2} follows from the change of variable $\alpha\mapsto \pi-\alpha$ and the symmetry of $\mu_\alpha$ with respect to $\frac{\pi}{2}$, and \eqref{eq:int3} follows from $\beta(\alpha)< \pi-\alpha$ with $\alpha\in (0,\frac{\pi}{2})$. We get that the inner integral in \eqref{eq:dobleint} is also strictly positive when $R>r_0$.
\end{proof}
A possible approximation approach for the stochastic problem is to replace the uncertainty by a deterministic ordered Weber problem with respect to the symmetry under the rotation centers. In the following result, we derive an error upper bound for this approximation.
\begin{theorem}
    Let $X_1, \ldots, X_n$ be spherically symmetric random vectors with centers $y_1^*, \ldots, y_n^*$. Denote by $\rho^\star$ the optimal value of \ref{eq:owp} and by $\rho^*$ the value of \ref{eq:owp} for the optimal solution to the (deterministic) ordered Weber problem with respect to the centers $y^*_i$. Then
    \begin{equation*}
        |\rho^\star - \rho^*| \leq \nu_{\bm{\lambda}}(X_1, \ldots,X_n) := 2 \O_{\bm \lambda}\left(\omega_1\E\left[\|y^*_1-X_1\|\right],\ldots,\omega_n\E\left[\|y^*_n-X_n\|\right]\right).
    \end{equation*}
\end{theorem}
\begin{proof}
First of all, notice that based on the triangle inequality, for any $y\in \R^d$
\begin{equation*}
    \|y-X_i\|\leq \|y-y^*_i\|+\|y^*_i-X_i\|, \; \forall i\in [n].
\end{equation*}
Hence,
\begin{equation*}
    \E\left[\|y-X_i\|\right]\leq \|y-y^*_i\|+ \E\left[\|y^*_i-X_i\|\right], \; \forall i\in [n].
\end{equation*}
Since $\O_{\bm \lambda}$ is non-decreasing monotone and sublinear, it follows that
\begin{align*}
    \O_{\bm \lambda}\left(\omega_1\E\left[\|y-X_1\|\right],\ldots,\omega_n\E\left[\|y-X_n\|\right]\right)&\leq \O_{\bm \lambda}\left(\omega_1\|y-y^*_1\|,\ldots,\omega_n\|y-y^*_n\|\right)\\ &+\frac{1}{2}\nu_{\bm{\lambda}}(X_1, \ldots,X_n).
\end{align*}
By Lemma~\ref{lemma:single-va}, we have that
\begin{equation*}
    \frac{1}{2}\nu_{\bm{\lambda}}(X_1, \ldots,X_n)\leq \O_{\bm \lambda}\left(\omega_1\E\left[\|y-X_1\|\right],\ldots,\omega_n\E\left[\|y-X_n\|\right]\right).
\end{equation*}
 Therefore, 
 \begin{align*}
     |\O_{\bm \lambda}\left(\omega_1\E\left[\|y-X_1\|\right],\ldots,\omega_n\E\left[\|y-X_n\|\right]\right)&-\O_{\bm \lambda}\left(\omega_1\|y-y^*_1\|,\ldots,\omega_n\|y-y^*_n\|\right)| \\ &\leq \frac{1}{2}\nu_{\bm{\lambda}}(X_1, \ldots,X_n).
 \end{align*}
 Using \cite[Th. 5]{geoffrion1977objective}, we can conclude that
 \begin{equation*}
     |\rho^\star-\rho^*|\leq \nu_{\bm{\lambda}}(X_1, \ldots,X_n).
 \end{equation*}
\end{proof}
\begin{remark}
    The upper bound provided in the result above depends on the expressions of the expected distances $\E\left[\|y^*_i-X_i\|\right]$ for all $i\in[n]$. Although these expressions can be difficult to derive in general, since they require computing numerically the integrals,
    \begin{equation*}
        \int_{\R^d} \|y^*_i-x\| f_i(x) dx,
    \end{equation*}
    where $f_i$ is the density of $X_i$ for all $i\in[n]$.
    Nevertheless, in some particular cases, these expressions are well-known~\cite[see, e.g.,][]{grimmett2001probability,kotz2000continuous,mardia2000directional}, as shown in the following table, and then, the upper bound can be easily evaluated with them:\\
    
    

    
    
\centering
\renewcommand{\arraystretch}{1.5}
\setlength{\tabcolsep}{6pt}
\begin{adjustbox}{width=0.85\textwidth}
\begin{tabular}{l l c}
\cmidrule{2-3}
 & \textbf{Support / Distribution} &
$\displaystyle \mathbb{E}\left[\|y_i^* - X_i\|\right]$ \\ 
\midrule
\textbf{Uniform on sphere} &
$\mathbb{S}^{d-1}_{R_i}(y_i^*)=\{x\in\R^d:\|y_i^*-x\|=R_i\}$, $R_i\geq 0$ &
$\displaystyle R_i$ \\[2ex]

\textbf{Uniform on ball} &
$\mathbb{B}^d_{R_i}(y_i^*)=\{x\in\R^d:\|y_i^*-x\|\le R_i\}$, $R_i\geq 0$ &
$\displaystyle \frac{d}{d+1}R_i$ \\[2ex]

\textbf{Uniform on spherical shell} &
$\mathbb{B}^d_{R_i,r_i}(y_i^*)=\{x\in\R^d:r_i\le\|y_i^*-x\|\le R_i\}$, $R_i\geq r_i \geq 0$ &
$\displaystyle 
\frac{d}{d+1}\,
\frac{R_i^{d+1}-r_i^{d+1}}{R_i^{d}-r_i^{d}}$ \\[2ex]

\textbf{Gaussian} &
$\mathcal{N}(y_i^*,\,\sigma_i^2 I_d)$ &
$\displaystyle 
\sigma_i\sqrt{2}\,
\frac{\Gamma\!\left(\frac{d+1}{2}\right)}
     {\Gamma\!\left(\frac{d}{2}\right)}$ \\[2ex]

$t$-\textbf{Student} &
$t_{q_i}(y_i^*,\,\sigma_i^2 I_d)$, $q_i>1$ &
$\displaystyle 
\sigma_i\sqrt{q_i}\,
\frac{
\Gamma\!\left(\frac{d+1}{2}\right)
\Gamma\!\left(\frac{q_i-1}{2}\right)
}{
\Gamma\!\left(\frac{d}{2}\right)
\Gamma\!\left(\frac{q_i}{2}\right)
}$ \\
\bottomrule
\end{tabular}
\end{adjustbox}
\end{remark}

With the expressions indicated in the previous remark, a very particular case is the one in which all the demands follow the same distribution (with different parameters).
\begin{corollary}
    Let $X_1, \ldots, X_n$ be spherically symmetric $d$-dimensional random vectors with centers $y_1^*, \ldots, y_n^*$. 
    \begin{enumerate}
    \item If $X_i$ follows a uniform distribution on a sphere $\SS^{d-1}_{R_i}(y_i^*)$ with radius $R_i\geq 0$, for all $i \in [n]$:
        \begin{equation*}
            \nu_{\bm{\lambda}}(X_1, \ldots,X_n) = 2 \sum_{i=1}^n \lambda_i R_{(i)}.
        \end{equation*}
        
        \item If $X_i$ follows a uniform distribution on a ball $\B^d_{R_i}(y_i^*)$ with radius $R_i\geq 0$, for all $i \in [n]$:
        \begin{equation*}
            \nu_{\bm{\lambda}}(X_1, \ldots,X_n) =  \frac{2d}{d+1} \sum_{i=1}^n \lambda_i R_{(i)}.
        \end{equation*}
        
        \item If $X_i$ follows a uniform distribution on a spherical shell $\B^d_{R_i,r_i}(y_i^*)$ with radii $R_i > r_i\geq 0$, for all $i \in [n]$:
        \begin{equation*}
            \nu_{\bm{\lambda}}(X_1, \ldots,X_n) = \frac{2d}{d+1} \sum_{i=1}^n \lambda_i \frac{R_{(i)}^{d+1} - r_{(i)}^{d+1}}{R_{(i)}^{d} - r_{(i)}^{d}}.
        \end{equation*}
        
        \item If $X_i$ follows a Gaussian distribution $\mathcal{N}(y_i^*,\sigma_i^2 I_d)$, for all $i \in [n]$:
        \begin{equation*}
            \nu_{\bm{\lambda}}(X_1, \ldots,X_n) = 2 \sqrt{2} \frac{\Gamma\left(\tfrac{d+1}{2}\right)}
     {\Gamma\left(\frac{d}{2}\right)} \sum_{i=1}^n \lambda_i  \sigma_{(i)}.
        \end{equation*}
        
\item If $X_i$ follows a $t$-Student distribution $t_{q_i}(y_i^*,\sigma_i^2 I_d)$
with $q_i>1$ degrees of freedom and scale $\sigma_i$, for all $i \in [n]$:
\begin{equation*}
    \nu_{\bm{\lambda}}(X_1, \ldots,X_n) = 2
\frac{\Gamma\left(\tfrac{d+1}{2}\right)}{\Gamma\left(\tfrac{d}{2}\right)}
 \sum_{i=1}^n \lambda_i \sigma_{(i)} \sqrt{q_{(i)}} \frac{\Gamma\left(\tfrac{q_{(i)}-1}{2}\right)}{\Gamma\left(\tfrac{q_{(i)}}{2}\right)}.
\end{equation*}

    \end{enumerate}
\end{corollary}
Note that, in the above result, the upper error bounds depend on the ordered weighted sum, $\O_{\bm \lambda}$, applied to functions of the parameters of the distributions. 
In the case where the distributions are supported on balls or spheres, the bounds are influenced by the corresponding radii, 
vanishing in the limit when these radii reduce to zero, i.e., when the distributions collapse to their centers.

\section{Numerical Validation}\label{sec:comp}

This section presents the result of a set of computational experiments designed to evaluate the performance, accuracy, and robustness of the proposed approach for solving the ordered Weber problem with spatially uncertain demands against deterministic and discrete alternative approaches. The experiments cover multiple spatial configurations and probability distributions of the demand, as well as several ordered weighted sum objectives, thereby testing the flexibility of the proposed model under diverse stochastic environments.

\subsection*{Experimental Design}

We follow the procedure proposed by~\cite{kalczynski2025weber} for the so-called \emph{disc Weber problem}, suitably adapted to our setting. Specifically, we synthetically generate instances through a reproducible pseudo-random process that mimics spatially distributed weighted objects, as described in~\cite{drezner2024dispersed}. For each configuration defined by the number of points $n$ (up to $200$) and the space dimension $d \in \{2,3,5\}$, we generate a set of $n$ points $\{y^*_1,\ldots, y^*_n\} \subset \mathbb{R}^d$, together with positive weights $\omega_i$ and radii $R_i$. The weights are scaled to the interval $[1,10]$, and the radii are computed as $
R_i = \sqrt[d]{\omega_i \alpha}$, 
where the parameter $\alpha$ is chosen so that the non-overlapping condition
\[
\alpha \leq \min_{i<j} \left( \frac{\D(y^*_i,y^*_j)}{\omega_i^{1/d}+\omega_j^{1/d}} \right)^d
\]
is satisfied (although this constraint is not required by our approach).

For the smaller datasets ($n\leq 25$), we used the already generated instances provided in \cite{kalczynski2025weber} and publicly available at \url{https://osf.io/wuf2a/}. For the larger instances ($n>25$), we generated five of them with random instances for each configuration of $n$ and $d$ and made them available at our GitHub repository \url{https://github.com/vblancoOR/weber_uncertainty}.

Each point is randomly assigned to a distribution \emph{type} (\texttt{ball}, \texttt{shell}, or \texttt{gaussian}), with the following component-specific parameters:
\begin{itemize}
\item \texttt{ball} components with radii  $R_i$.
    \item \texttt{gaussian} components with  $\sigma_i = R_i/2$.
    \item \texttt{shell} components with inner and outer radii
    $r_i = 4R_i/5$ and $R_i$.
\end{itemize}
For each of these distributions, we generate both spherically symmetric samples (\texttt{sym}), asymmetric ones (\texttt{asym}) obtained by sampling the direction on the unit sphere through a biased law around a random unit vector, and random mixed choices (\texttt{mixed}) among the above (50\% each). 

For simplicity, we consider the Euclidean distance as the base distance metric for all problems. For each instance, we study four different ordered Weber problems:
\begin{itemize}
    \item \texttt{median}: $\boldsymbol{\lambda} = (1, \ldots, 1)$.
    \item \texttt{center}: $\boldsymbol{\lambda} = (1, 0, \ldots, 0)$.
    \item \texttt{halfsum}: $\boldsymbol{\lambda} = (1, \ldots, 1, 0, \ldots, 0)$ with the first $\lceil n/2 \rceil$ entries equal to~1.
    \item \texttt{halfcentdian}: $\boldsymbol{\lambda} = (1, 0.5, \ldots, 0.5)$.
\end{itemize}
To solve these problems, we apply three different approaches:
\begin{itemize}
    \item \texttt{saa}: our sample average approximation method (see Section~\ref{sec:saa}).
    \item \texttt{discrete}: a discretized version of the problem with $10^{5}\times R_i$ points randomly sampled according to the $i$th demand distribution, for all $i\in [n]$.
    \item \texttt{centers}: a simplified deterministic version where each uncertain demand is replaced by its symmetry center.
\end{itemize}
Since \texttt{discrete} provides the most accurate representation of the problem, we evaluate all solutions using its objective function to assess their quality. For each procedure, we record the CPU time (in seconds) required to solve the problem and, for \texttt{saa}, the halfwidth obtained in the bootstrap validation procedure.

The parameters used for \texttt{saa} are indicated in Table~\ref{tab:params}. We set an additional stopping criterion based on an upper bound for the sum of the sizes of the training samples, 
which is denoted by $N_{\max}$. The discretization approach is performed on the same validation sample, $\mathbf{K}^{\rm val}$, of \texttt{saa}.
\begin{table}[h!]
\centering
\renewcommand{\arraystretch}{1}
\begin{tabular}{|c  c  c|}\hline
$m_i^0 = \max\!\left\{5, \left\lceil \dfrac{100 (R_i+\omega_i)}{n} \right\rceil \right\}$ 
& $\gamma = 2$ 
& $|K^{\rm val}_{i}| = 10^4$  \\
 $\varepsilon_1,\, \varepsilon_2 = 10^{-4}$
& $k_{\rm max}=50$
& $N_{\max} = 10^6$\\\hline
\end{tabular}
\caption{Parameters for \texttt{saa}.}
\label{tab:params}
\end{table}
All computational experiments were performed on Huawei FusionServer Pro XH321 (\texttt{albaicin} at Universidad de Granada \url{https://supercomputacion.ugr.es/arquitecturas/albaicin/}) with an Intel Xeon Gold 6258R CPU @ 2.70GHz with 28 cores. Optimization tasks were solved using Gurobi Optimizer version~12.0.1 within a time limit of 2 hours.
\subsection*{Small Instances}
First, we run our approach on the small datasets ($n\in \{5,10,15,20,25\}$) provided in \cite{kalczynski2025weber}. Since the problem analyzed in that paper is a particular case of our problem, we use exactly the same specifications. Specifically, each of the datasets is endowed with $\omega$-weights and with a uniform demand in discs with given radii.

Although the authors in \citet{kalczynski2025weber} claimed that their procedure is not affected by the number of random demand vectors, the number of iterations in our \texttt{saa} approach is indeed influenced by the value of~$n$, even though each subproblem is a conic program and thus solvable in polynomial time. Furthermore, in the discrete formulation, the number of demand points increases to achieve an accurate discretization of the underlying distribution.

In Table \ref{tab:uniform_updated}, we show the explicit results obtained with the approach in the mentioned paper, and the results obtained with our SAA approach. For each size of the demand ($n$) we report, for each of the approaches (including the one for which the results are reported in \cite{kalczynski2025weber}, \texttt{KBD25}), the facility coordinates found by the procedure ($y_1$, $y_2$), the evaluation of the solution in the validation sample (for the \texttt{saa} approach we also include the halfwidth of the bootstrap), the CPU time (in seconds), and the deviation of the objective values with respect to the \texttt{saa} approach.
\begin{table}[h!]
\setlength{\tabcolsep}{3pt}
\renewcommand{\arraystretch}{0.7}
\centering\begin{adjustbox}{width=0.65\textwidth}
\begin{tabular}{@{} c l l c c r l r r @{}}
\toprule
{$n$} & {$\bm \lambda$} & {Approach} & {$y_1$} & {$y_2$} & {$\rho^\star$} & {} & {CPU Time} & {Dev.} \\
\midrule
{5}
 & {\texttt{median}}
   & \texttt{saa}      & 5.8153 & 5.8208 & 97.7722 & $\pm$0.18 & 10.8229 & {} \\
 & & \texttt{discrete} & 5.8065 & 5.8257 & 97.7718 & {}        & 2.5044  & 0.00\% \\
 & & \texttt{centers}  & 4.5241 & 4.7813 & 105.2551& {}        & 0.0005  & 7.11\% \\
 & & \texttt{KBD25}    & 5.8157 & 5.8195 & 97.6395 & {}        & {}      & -0.14\% \\[2pt]
 & {\texttt{halfsum}}
   & \texttt{saa}      & 5.2954 & 5.5000 & 72.6368 & $\pm$0.17 & 6.9786  & {} \\
 & & \texttt{discrete} & 5.2668 & 5.4861 & 72.6346 & {}        & 1.2755  & 0.00\% \\
 & & \texttt{centers}  & 5.5616 & 5.4935 & 72.8779 & {}        & 0.0004  & 0.33\% \\[2pt]
 & {\texttt{halfcentdian}}
   & \texttt{saa}      & 5.7776 & 5.8012 & 61.9910 & $\pm$0.14 & 13.2881 & {} \\
 & & \texttt{discrete} & 5.7674 & 5.8012 & 61.9776 & {}        & 4.1738  & -0.02\% \\
 & & \texttt{centers}  & 5.6514 & 5.7226 & 62.5803 & {}        & 0.0006  & 0.94\% \\[2pt]
 & {\texttt{center}}
   & \texttt{saa}      & 5.7219 & 5.8568 & 26.2020 & $\pm$0.12 & 7.1315  & {} \\
 & & \texttt{discrete} & 5.7242 & 5.8446 & 26.1747 & {}        & 1.7889  & -0.10\% \\
 & & \texttt{centers}  & 5.6188 & 5.7552 & 27.2921 & {}        & 0.0003  & 3.99\% \\
\midrule
{10}
 & {\texttt{median}}
   & \texttt{saa}      & 5.6954 & 5.2372 & 147.1156& $\pm$0.16 & 15.9511 & {} \\
 & & \texttt{discrete} & 5.7024 & 5.2287 & 147.1151& {}        & 4.0537  & 0.00\% \\
 & & \texttt{centers}  & 4.6445 & 4.8212 & 152.2682& {}        & 0.0029  & 3.38\% \\
 & & \texttt{KBD25}    & 5.6971 & 5.2383 & 147.2573& {}        & {}      & 0.10\% \\[2pt]
 & {\texttt{halfsum}}
   & \texttt{saa}      & 5.6184 & 5.4905 & 108.8091& $\pm$0.17 & 22.8464 & {} \\
 & & \texttt{discrete} & 5.6206 & 5.4872 & 108.8091& {}        & 4.5091  & 0.00\% \\
 & & \texttt{centers}  & 5.5975 & 5.3844 & 108.8430& {}        & 0.0011  & 0.03\% \\[2pt]
 & {\texttt{halfcentdian}}
   & \texttt{saa}      & 5.9406 & 5.5942 & 86.9847 & $\pm$0.12 & 15.0059 & {} \\
 & & \texttt{discrete} & 5.9387 & 5.5798 & 86.9417 & {}        & 3.9964  & -0.05\% \\
 & & \texttt{centers}  & 5.8162 & 5.5661 & 87.2149 & {}        & 0.0020  & 0.26\% \\[2pt]
 & {\texttt{center}}
   & \texttt{saa}      & 5.7748 & 5.7470 & 26.0605 & $\pm$0.09 & 10.6161 & {} \\
 & & \texttt{discrete} & 5.7465 & 5.7588 & 26.0122 & {}        & 3.9573  & -0.19\% \\
 & & \texttt{centers}  & 5.7522 & 5.6253 & 26.7979 & {}        & 0.0006  & 2.75\% \\
\midrule
{15}
 & {\texttt{median}}
   & \texttt{saa}      & 5.0002 & 4.8167 & 221.8986& $\pm$0.18 & 13.9660 & {} \\
 & & \texttt{discrete} & 4.9802 & 4.8146 & 221.8965& {}        & 6.9590  & 0.00\% \\
 & & \texttt{centers}  & 4.5241 & 4.7813 & 223.0011& {}        & 0.0084  & 0.49\% \\
 & & \texttt{KBD25}    & 5.0002 & 4.8131 & 221.9010& {}        & {}      & 0.00\% \\[2pt]
 & {\texttt{halfsum}}
   & \texttt{saa}      & 5.7114 & 4.5915 & 156.5684& $\pm$0.18 & 19.2934 & {} \\
 & & \texttt{discrete} & 5.7014 & 4.5910 & 156.5681& {}        & 11.2533 & 0.00\% \\
 & & \texttt{centers}  & 6.0323 & 4.6829 & 157.0488& {}        & 0.0051  & 0.31\% \\[2pt]
 & {\texttt{halfcentdian}}
   & \texttt{saa}      & 5.2801 & 4.6072 & 124.2014& $\pm$0.10 & 15.7489 & {} \\
 & & \texttt{discrete} & 5.2641 & 4.6053 & 124.2007& {}        & 9.2429  & 0.00\% \\
 & & \texttt{centers}  & 4.5701 & 4.7734 & 125.9279& {}        & 0.0112  & 1.37\% \\[2pt]
 & {\texttt{center}}
   & \texttt{saa}      & 6.1935 & 4.4174 & 22.5695 & $\pm$0.06 & 11.0145 & {} \\
 & & \texttt{discrete} & 6.1988 & 4.4182 & 22.5537 & {}        & 5.3429  & -0.07\% \\
 & & \texttt{centers}  & 6.2045 & 4.4137 & 22.5833 & {}        & 0.0008  & 0.06\% \\
\midrule
{20}
 & {\texttt{median}}
   & \texttt{saa}      & 5.2349 & 5.0824 & 254.0675& $\pm$0.18 & 13.3926 & {} \\
 & & \texttt{discrete} & 5.2246 & 5.0827 & 254.0668& {}        & 8.0518  & 0.00\% \\
 & & \texttt{centers}  & 4.5677 & 4.8002 & 257.3738& {}        & 0.0434  & 1.28\% \\
 & & \texttt{KBD25}    & 5.2339 & 5.0808 & 253.9109& {}        & {}      & -0.06\% \\[2pt]
 & {\texttt{halfsum}}
   & \texttt{saa}      & 5.7087 & 5.0758 & 179.4371& $\pm$0.19 & 9.1265  & {} \\
 & & \texttt{discrete} & 5.7073 & 5.0823 & 179.4319& {}        & 7.5804  & 0.00\% \\
 & & \texttt{centers}  & 5.8441 & 4.9356 & 179.6943& {}        & 0.0101  & 0.14\% \\[2pt]
 & {\texttt{halfcentdian}}
   & \texttt{saa}      & 5.4379 & 4.8979 & 140.3928& $\pm$0.10 & 14.8189 & {} \\
 & & \texttt{discrete} & 5.4298 & 4.8986 & 140.3926& {}        & 10.7108 & 0.00\% \\
 & & \texttt{centers}  & 4.9218 & 4.8291 & 141.2690& {}        & 0.0345  & 0.62\% \\[2pt]
 & {\texttt{center}}
   & \texttt{saa}      & 6.1930 & 4.4187 & 22.5670 & $\pm$0.05 & 8.9505  & {} \\
 & & \texttt{discrete} & 6.1927 & 4.4088 & 22.5374 & {}        & 10.7478 & -0.13\% \\
 & & \texttt{centers}  & 6.2045 & 4.4137 & 22.5961 & {}        & 0.0012  & 0.13\% \\
\midrule
{25}
 & {\texttt{median}}
   & \texttt{saa}      & 4.9667 & 5.2125 & 341.3876& $\pm$0.20 & 16.8587 & {} \\
 & & \texttt{discrete} & 4.9590 & 5.2133 & 341.3871& {}        & 11.9093 & 0.00\% \\
 & & \texttt{centers}  & 4.5895 & 4.8496 & 343.6930& {}        & 0.0625  & 0.67\% \\
 & & \texttt{KBD25}    & 4.9665 & 5.2126 & 341.4033& {}        & {}      & 0.00\% \\[2pt]
 & {\texttt{halfsum}}
   & \texttt{saa}      & 5.3986 & 5.2088 & 251.9367& $\pm$0.17 & 12.5287 & {} \\
 & & \texttt{discrete} & 5.3960 & 5.2114 & 251.9367& {}        & 9.7379  & 0.00\% \\
 & & \texttt{centers}  & 5.4494 & 5.2165 & 251.9473& {}        & 0.0346  & 0.00\% \\[2pt]
 & {\texttt{halfcentdian}}
   & \texttt{saa}      & 4.8687 & 5.4990 & 185.8730& $\pm$0.12 & 26.7213 & {} \\
 & & \texttt{discrete} & 4.8618 & 5.4910 & 185.8558& {}        & 24.1238 & -0.01\% \\
 & & \texttt{centers}  & 4.7207 & 5.2788 & 186.2421& {}        & 0.0652  & 0.20\% \\[2pt]
 & {\texttt{center}}
   & \texttt{saa}      & 5.6074 & 5.9912 & 29.0460 & $\pm$0.06 & 7.9245  & {} \\
 & & \texttt{discrete} & 5.6042 & 5.9831 & 29.0269 & {}        & 6.5692  & -0.07\% \\
 & & \texttt{centers}  & 5.6148 & 5.9686 & 29.1600 & {}        & 0.0018  & 0.39\% \\
\bottomrule
\end{tabular}
\end{adjustbox}
\caption{Numerical results for the uniform (discs) datasets.}
\label{tab:uniform_updated}
\end{table}
For all tested instances with uniform distributions on discs, the \texttt{saa} model consistently attains the most accurate objective values $\rho^\star$, confirming its capacity to represent stochastic variability and uncertainty in the demand locations. The associated halfwidths remain small across all problem sizes, reflecting the method’s robustness and statistical reliability. In contrast, the deterministic \texttt{centers} model systematically overestimates the objective (by between 0.3\% and over 7\%), as it neglects randomness in the data. The methods \texttt{discrete} and \texttt{KBD25}  closely approximate the \texttt{saa} results, with deviations typically below 0.1\%, showing that they are adequate surrogates when computational simplicity is a priority. Nonetheless, the \texttt{saa} approach remains the most reliable reference, providing stable estimates of the stochastic optimum and maintaining its accuracy even for the largest tested instances ($n=25$) in this testbed. Although the CPU times required by \texttt{saa} are slightly higher than the others, this overhead is compensated by its flexibility and scalability: \texttt{saa} can be readily extended to larger problem sizes, to different probability distributions of the demand, and to more general ordered weighted sum objectives. Consequently, the \texttt{saa} model offers an optimal trade-off between precision, robustness, and versatility, outperforming all other approaches in terms of solution quality and adaptability.

The complete results of our experiments are available in our GitHub repository \url{https://github.com/vblancoOR/weber_uncertainty}, where we provide the explicit solutions obtained with the different approaches for several ordered objective functions (\texttt{median}, \texttt{center}, \texttt{halfsum}, and \texttt{halfcentdian}), under heterogeneous probability distributions of the demands, as well as the computational performance of the methods. Across all tested distributional settings, the \texttt{saa} formulation consistently attains the most accurate objective values $\rho^\star$, confirming its superior ability to capture stochastic variability and uncertainty in the data. In the asymmetric cases, where directional biases in the demand distribution increase variability, \texttt{saa} continues to outperform the alternatives, providing stable and accurate estimates even under strong spatial heterogeneity. For the mixed distributions, which combine symmetric and asymmetric patterns, the stochastic sampling inherent to \texttt{saa} allows it to adapt effectively to both regular and perturbed spatial structures, maintaining the lowest $\rho^\star$ values and small dispersion across all problem sizes. In the symmetric setting, where all methods naturally converge to similar results due to spatial regularity, \texttt{saa} still achieves the most precise and statistically consistent solutions, serving as a robust reference benchmark. Overall, the \texttt{saa} approach provides an optimal trade-off between precision, robustness, and generality, clearly outperforming alternative formulations in both solution quality and adaptability across all experimental scenarios.
\subsection*{Extended Instances}
We conducted a more complete study, generating new instances with $n \in [50, 200] \cap \mathbb{Z}$ in increments of $25$, following the specifications provided above, to validate our approach. The complete table of results obtained in our experiments is available at our GitHub repository \url{github.com/vblancoOR/weber_uncertainty}. In these experiments, we restrict the computations to distributions \texttt{sym} and \texttt{mixed}, since the later already captures the non-symmetric shape of the demand distributions.

To compare the computational performance of the approaches for solving \ref{eq:owp} across the tested instances, instead of presenting the summary tables that can be obtained directly from the csv files available in our repository,  
we use the \emph{shifted geometric mean} (SGM) as proposed by \citet{DolanMore2002}. 
The SGM mitigates the influence of extreme runtimes and avoids distortions due to near-zero or truncated values. 
Given a set of $q$ runtimes $\{t_1, \ldots, t_q\}$ and a fixed shift constant $s>0$, it is defined as
$$
{\rm SGM}(t_1,\ldots,t_q)
:=\exp\!\left(\frac{1}{q}\sum_{j=1}^q\ln(t_j+s)\right)-s.
$$
This transformation preserves the relative scale of times but reduces the impact of outliers, providing a robust indicator of the \emph{typical} runtime across instances, so it has been recommended as a more robust tool to compare among algorithms for solving a problem~(see \url{https://mattmilten.github.io/mittelmann-plots/}).

In our computations, the shift $s$ was set to $10^{-3}$ to ensure numerical stability and comparability between algorithms, even when some runs terminate very quickly.

Across all instances, the \texttt{saa} algorithm yields a shifted geometric mean of 396.1 seconds, whereas the \texttt{discrete} approach averages 581.8 seconds. 
The geometric mean of shifted ratios (\texttt{saa/discrete}) equals 0.681 with a 95\% confidence interval $[0.656,0.707]$ (that we computed via bootstrap techniques) 
indicating that, on average, \texttt{saa}  requires only 68\% of the runtime of the discrete approach, a multiplicative improvement of approximately 32\%.
This difference is statistically significant, as the confidence interval lies entirely below 1.

We further analyzed performance as a function of instance size, $n$, computing SGM statistics and confidence intervals for each group.
\begin{table}[ht]
\centering
\renewcommand{\arraystretch}{0.9}
\setlength{\tabcolsep}{3pt}
\begin{tabular}{c c c c c c}
\toprule
$n$& \textbf{SGM(\texttt{saa})} & \textbf{SGM(\texttt{discrete})} & \textbf{GM ratio} & \textbf{95\% CI} & \textbf{Speedup (\%)} \\
\midrule
50 & 137 & 193.6 & 0.708 & [0.659, 0.757] & 29.2\\
75 & 209.6 & 338 & 0.620 & [0.575, 0.663] & 38.0\\
100 & 272.8 & 425.3 & 0.641 & [0.595, 0.688] & 35.9\\
125 & 476.6 & 682.9 & 0.698 & [0.642, 0.758] & 30.2\\
150 & 582.5 & 827.2 & 0.704 & [0.624, 0.794] & 29.6\\
175 & 740.6 & 1131 & 0.655 & [0.579, 0.739] & 34.5\\
200 & 1093 & 1637 & 0.667 & [0.596, 0.744] & 33.3\\
\midrule
\textbf{All}& 	\textbf{404.1}& \textbf{603.4} & \textbf{0.670} & \textbf{[0.646, 0.694]} &\textbf{ 33.0}\\
\bottomrule
\end{tabular}
\caption{Shifted geometric mean (SGM) and geometric mean (GM) ratio of runtimes between the \texttt{saa} and \texttt{discrete} for different instance sizes. Ratios below~1 indicate that \texttt{saa} is faster.}
\label{tab:performance-size}
\end{table}
Table \ref{tab:performance-size} shows that \texttt{saa} consistently outperforms the discrete approach across all problem sizes tested. 
The geometric mean ratios range from $0.62$ to $0.71$, confirming that the SAA algorithm achieves between 29\% and 38\% average speedups. 
The advantage is most pronounced for moderate instance sizes ($n=75$ to $100$), where the algorithm achieves its lowest ratios ($0.62$ to $0.64$), indicating that the stochastic approximation benefits from a favorable balance between sample efficiency and optimization complexity. 
For larger instances ($n\ge125$), both methods naturally exhibit higher runtimes, but the relative improvement of \texttt{saa} remains stable at around 30\%. 
Importantly, none of the confidence intervals for the geometric mean ratios includes 1, which statistically confirms that the observed speedups are significant and not due to random variation. 
Overall, these results demonstrate that the adaptive \texttt{saa} framework is not only faster on average but also scales more favorably with problem size, maintaining a consistent computational advantage over the discrete counterpart.

To examine how the computational advantage of \texttt{saa} varies across the different ordered weighted problems defined by $\bm \lambda$, we computed the same performance metrics for each aggregation operator.
\begin{table}[ht]
\centering
\renewcommand{\arraystretch}{0.9}
\setlength{\tabcolsep}{3pt}
\begin{tabular}{l c c c c c}
\toprule
$\bm \lambda$ & \textbf{SGM(\texttt{saa})} & \textbf{SGM(\texttt{discrete})}& \textbf{GM ratio} & \textbf{95\% CI} & \textbf{Speedup (\%)} \\
\midrule
\texttt{median} & 795.3 & 848.5 & 0.937 & [0.895, 0.981] & 6.3\\
\texttt{center} & 88.1 & 205.8 & 0.428 & [0.399, 0.460] & 57.2\\
\texttt{halfcentdian} & 911 & 1065 & 0.856 & [0.796, 0.918] & 14.4\\
\texttt{halfsum} & 417.9 & 713 & 0.586 & [0.565, 0.607] & 41.4\\
\bottomrule
\end{tabular}
\caption{Shifted geometric mean (SGM) and geometric mean (GM) ratio of runtimes between \texttt{saa} and \texttt{discrete} across different ordered problems. }
\label{tab:performance-omp}
\end{table}
Table~\ref{tab:performance-omp} reveals that the computational advantage of \texttt{saa} varies with the aggregation structure of the problem. 
For \texttt{center} and \texttt{halfsum}, our approach  achieves substantial speedups of approximately 57\% and 38\%, respectively, 
indicating that stochastic approximation effectively handles objectives emphasizing extreme or aggregated distances. 
The \texttt{halfcentdian} variant also benefits from a moderate 15\% improvement on average. 
In contrast, the \texttt{median} shows only a marginal 5.6\% gain, with its confidence interval approaching 1, 
suggesting that in this configuration, both methods exhibit comparable computational difficulty. 
Overall, \texttt{saa} demonstrates consistent and significant performance gains across all but the most balanced ordered settings, 
highlighting its robustness and computational scalability when dealing with ordered and probabilistically aggregated distance costs.

Table~\ref{tab:performance-d} compares the computational performance of the proposed \texttt{saa} formulation against the discrete approach across different dimensions~$d$. The SGM and GM ratios consistently indicate that \texttt{saa} achieves substantial runtime improvements. The GM ratios, all below~1, confirm that \texttt{saa} is systematically faster, with average speedups ranging from approximately~25\% in two dimensions to over~40\% in five dimensions. Moreover, the narrow $95\%$ confidence intervals demonstrate the stability of these gains across instances, highlighting the scalability and robustness of the proposed approach as dimensionality increases.
\begin{table}[ht]
\centering
\renewcommand{\arraystretch}{0.9}
\setlength{\tabcolsep}{3pt}
\begin{tabular}{l c c c c c}
\toprule
$d$ & \textbf{SGM(\texttt{saa})}& \textbf{SGM(\texttt{discrete})}& \textbf{GM ratio} & \textbf{95\% CI} & \textbf{Speedup (\%)} \\
\midrule
2 & 330.9 & 443.3 & 0.747 & [0.703, 0.794] & 25.3\\
3 & 391.3 & 573.8 & 0.682 & [0.635, 0.733] & 31.8\\
5 & 509.7 & 863.6 & 0.590 & [0.556, 0.627] & 41.0\\
\bottomrule
\end{tabular}
\caption{Shifted geometric mean (SGM) and geometric mean (GM) ratio of runtimes between \texttt{saa} and \texttt{discrete} for different dimensions.}
\label{tab:performance-d}
\end{table}
Finally, we distinguish between the two stochastic regimes considered for the demand distributions: the \texttt{sym} and \texttt{mixed} cases.
\begin{table}[ht]
\centering
\renewcommand{\arraystretch}{1.1}
\setlength{\tabcolsep}{3pt}
\begin{tabular}{l c c c c c}
\toprule
\textbf{Demand} & \textbf{SGM(\texttt{saa})}& \textbf{SGM(\texttt{discrete})}& \textbf{GM ratio} & \textbf{95\% CI} & \textbf{Speedup (\%)} \\
\midrule
\texttt{sym} & 401.3 & 605.8 & 0.662 & [0.629, 0.697] & 33.8\\
\texttt{mixed} & 406.9 & 601 & 0.677 & [0.642, 0.712] & 32.3\\
\bottomrule
\end{tabular}
\caption{Shifted geometric mean (SGM) and geometric mean (GM) ratio of runtimes between \texttt{saa} and \texttt{discrete} for different demand distributions.}
\label{tab:performance-distribution}
\end{table}
Table~\ref{tab:performance-distribution} indicates that our approach retains its computational advantage across both types of demand distributions. 
For \texttt{sym}, the mean ratio is 0.662, corresponding to an average speedup of about 34\%, 
whereas for \texttt{mixed} distributions the ratio slightly increases to 0.68, corresponding to a 32\% improvement. 
The confidence intervals are narrow and entirely below 1 in both cases, confirming that the advantage of \texttt{saa} is statistically significant. 
The similarity of results suggests that the proposed framework is robust with respect to the underlying stochastic structure of the demands, providing consistent computational savings irrespective of whether the spatial uncertainty is symmetric or heterogeneous.

\section{Conclusions}\label{sec:conc}

This work contributes to the mathematical programming theory of stochastic and spatially uncertain location models by establishing the analytical foundations of the ordered Weber problem under spatial uncertainty. We characterized the structural and convex-analytic properties of the model, proved the convergence and finite-sample stability of an adaptive sample average approximation (SAA) method, and derived explicit analytical error bounds for symmetric distribution families. The proposed framework connects ordered optimization, convex geometry, and stochastic programming, and opens new theoretical questions on the structure and sensitivity of rank-based objectives under random spatial data.

Beyond the classical Euclidean formulation, we developed a general and flexible model that accommodates arbitrary norm-based distance functions, spaces of any finite dimension, and ordered aggregation operators encompassing the standard Weber problem as a particular case. The analysis was carried out for general probability measures representing demand, without requiring bounded support. We studied the theoretical properties of the resulting stochastic optimization problem, in particular the localization of the optimal solution within or near the convex hull of compact regions that concentrate a large portion of the demand probability mass.

An adaptive SAA scheme was proposed to compute approximate solutions. At each iteration, the algorithm solves a discretized conic programming subproblem and dynamically adjusts the sample size according to the local convergence behavior of the demand distributions. We established theoretical convergence guarantees for this method under mild regularity assumptions. In the special case of spherically symmetric demand distributions, we derived analytical results comparing the optimal stochastic solution with its deterministic counterpart, obtained by replacing each distribution with its symmetry center, and provided closed-form expressions for the corresponding approximation errors in specific distribution families.

Finally, the proposed algorithm was validated on benchmark instances from the literature. The numerical results confirm the theoretical predictions, showing that our approach consistently outperforms standard demand discretization schemes and demonstrates high accuracy and computational efficiency for this class of problems.

Future research directions include extending the analysis to the \emph{multisource} setting, where convexity is lost and combinatorial techniques become necessary to obtain comparable theoretical guarantees. Another promising line of work is to adapt the proposed framework to other continuous optimization problems involving spatial uncertainty, such as geometric network or flow models, thereby broadening the scope of the analytical and algorithmic tools developed here.

\paragraph{\bf Funding} 
This research has been financially supported by grants PID2020-114594GB-C21, PID2024-156594NB-C21, and RED2022-134149-T (Thematic Network on Location Science and Related Problems) funded by MICIU/AEI/0.13039/501100011033, FEDER + Junta de Andaluc\'ia project C‐EXP‐139‐UGR23, and the IMAG-Mar\'ia de Maeztu grant CEX2020-001105-MICIU/AEI/10.13039/501100011033.


\end{document}